\documentclass[a4, 12pt]{amsart}
\usepackage{mathrsfs}
\usepackage{amsmath}
\usepackage{amsfonts}
\usepackage{amssymb}
\usepackage{amsmath,amssymb,amsthm}
\usepackage{amsmath,amssymb,amsthm,amscd}
\usepackage[frame,cmtip,arrow,matrix,line,graph,curve]{xy}
\usepackage{graphpap, color}
\usepackage[mathscr]{eucal}
\usepackage{color}
\usepackage{verbatim}
\usepackage[colorlinks, linkcolor=red, anchorcolor=blue,citecolor=blue]{hyperref}
\usepackage{cite}

\numberwithin{equation}{section} \numberwithin{equation}{section}
\newtheorem{thm}{Theorem}[section]
\newtheorem{lem}[{thm}]{Lemma}

\newtheorem{corr}[{thm}]{Corollary}
\newtheorem{rem}{Remark}[section]{\bfseries\upshape}

\newtheorem*{prob}{Problem}

\newtheorem*{Nash}{Nash's Theorem}

\newcommand{\DOI}[1]{doi: \href{https://doi.org/#1}{#1}}

\renewcommand{\oddsidemargin}{5mm}

\makeatletter
\title[Eigenvalue Inequalities of $\mathfrak{L}^{2}_{\nu}$ Operator ]{Eigenvalues for the Clamped Plate Problem of \\ $\mathfrak{L}^{2}_{\nu}$ Operator on Complete  Riemannian manifolds}
\author[L. Zeng ]{Lingzhong Zeng}

\address{Lingzhong Zeng
\\  \newline \indent School of Mathematics and Statistics
\\  \newline \indent   Jiangxi Normal University, Nanchang 330022,  China. lingzhongzeng@yeah.net}

\begin{document}
\maketitle

\begin{abstract} $\mathfrak{L}_{\nu}$ operator is an important extrinsic differential operator of divergence type and has profound geometric settings.  In this paper, we consider the clamped plate problem of $\mathfrak{L}^{2}_{\nu}$ operator on a bounded domain of the complete Riemannian manifolds.  A general formula  of eigenvalues of $\mathfrak{L}^{2}_{\nu}$ operator is established. Applying this general formula, we obtain some estimates for the eigenvalues with higer order on the complete Riemannian manifolds.
As several fascinating applications, we discuss this eigenvalue problem on the complete translating solitons, minimal submanifolds on the Euclidean space, submanifolds on the unit sphere and projective spaces. In particular, we get a universal inequality with respect to the $\mathcal{L}_{II}$ operator on the translating solitons.  Usually, it is very difficult to get universal inequalities for weighted Laplacian and even Laplacian on the complete Riemannian manifolds. Therefore, this work can be viewed as a new contribution to universal inequality. \end{abstract}

\footnotetext{{\it Key words and phrases}: mean curvature flows; $\mathfrak{L}_{\nu}^{2}$  operator; clamped plate problem;
eigenvalues; Riemannian manifolds; translating solitons.} \footnotetext{2010
\textit{Mathematics Subject Classification}:
 35P15, 53C40.}

\footnotetext{The research was partially supported by the National Natural
Science Foundation of China (Grant Nos. 11861036 and 11826213) and   Natural Science Foundation of Jiangxi Province (Grant No. 20171ACB21023).}

\section{Introduction}
Suppose that $\mathcal{M}^{n}$ is an $n$-dimensional, complete Riemannian manifold $\mathcal{M}^{n}$ and $\Omega$ is a bounded domain with piecewise smooth boundary $\partial\Omega$. We consider  the fixed membrane problem of Laplacian on $\Omega\subset\mathcal{M}^{n}$:

\begin{equation}\label{Laplace-prob} {\begin{cases} \
 \Delta u  =-\lambda u , \ \ & {\rm in} \ \ \ \ \Omega, \\
  \ u=0, \ \ & {\rm on} \ \ \partial \Omega,
\end{cases}}\end{equation}where $\Delta$ denotes the Laplacian on the Riemannian manifold  $\mathcal{M}^{n}$.
Let $\lambda_{k}$ denote the $k^{th}$ eigenvalue, and then the spectrum of the eigenvalue problem \eqref{Laplace-prob} is
discrete and satisfies
\begin{equation*}
0<\lambda_{1}\leq\lambda_{2}\leq\cdots\leq\lambda_{k}\leq\cdots\rightarrow+\infty,
\end{equation*}
where each eigenvalue is repeated under counting its multiplicity.
When $\mathcal{M}^{n}$ is an $n$-dimensional Euclidean space $\mathbb{R}^{n}$, Payne,
P\'{o}lya and Weinberger \cite{PPW2} studied Dirichlet eigenvalue problem \eqref{Laplace-prob}, and obtained a universal inequality as follows:

\begin{equation}\label{ppw-ineq}\lambda_{k+1}-\lambda_{k}\leq\frac{4}{nk}\sum^{k}_{i=1}\lambda_{i}.\end{equation} Here, the words ``universal" means that eigenvalue inequalities are independent of the domain.
Furthermore, in various settings, many mathematicians extended the universal inequality given by Payne, P\'{o}lya and Weinberger.  In particular, Hile and Protter \cite{HP} proved the following universal inequality of eigenvalues:

\begin{equation}\label{hp-ineq}\sum^{k}_{i=1}\frac{\lambda_{i}}{\lambda_{k+1}-\lambda_{i}}\geq\frac{nk}{4},\end{equation}
which is sharper than eigenvalues inequality \eqref{ppw-ineq}. Furthermore, an amazing contribution to eigenvalue inequality is due to Yang \cite{Y} (cf. \cite{CY2}). He obtained a very sharp universal inequality:

\begin{equation}\label{y1-ineq}\sum^{k}_{i=1}(\lambda_{k+1}-\lambda_{i})^{2}\leq\frac{4}{n}\sum^{k}_{i=1}(\lambda_{k+1}-\lambda_{i})\lambda_{i}.\end{equation}
From \eqref{y1-ineq}, one can infer that

\begin{equation}\label{y2-ineq}\lambda_{k+1}\leq\frac{1}{k}(1+\frac{4}{n})\sum^{k}_{i=1}\lambda_{i}.\end{equation}
The inequalities \eqref{y1-ineq} and \eqref{y2-ineq} are called by Ashbaugh Yang's first inequality and second inequality,
respectively (cf. \cite{A1}, \cite{A2}). Indeed, according to Chebyshev's inequality, we have the following connections
$$\eqref{y1-ineq}\Rightarrow \eqref{y2-ineq} \Rightarrow \eqref{hp-ineq} \Rightarrow \eqref{ppw-ineq}.$$

To describe vibrations of a clamped plate in elastic mechanics, one usually
consider the following Dirichlet eigenvalue problem for biharmonic operator:

\begin{equation}\label{clamped}\left\{\begin{array}{l}
\Delta^{2} u=\Lambda u, \quad \text { in } \Omega, \\
u=\frac{\partial u}{\partial \textbf{n}}=0,\quad \text { on }\partial \Omega,
\end{array}\right.\end{equation}
where $\Delta$ is the Laplacian in $\mathbb{R} ^{n}$ and $\Delta^{2}$ is the biharmonic operator in $\mathbb{R} ^{n}$, and this eigenvalue problem is called a clamped
plate problem.
For this clamped plate problem, in $1956,$ Payne, P\'{o}lya and Weinberger \cite{PPW2} established an eigenvalue inequality. They obtained a universal inequality as follows:
\begin{equation}\label{PPW-ine-2}
\Lambda_{k+1}-\Lambda_{k} \leq \frac{8(n+2)}{n^{2}} \frac{1}{k} \sum_{i=1}^{k} \Lambda_{i}.
\end{equation}
As a generalization of their result, in $1984,$ Hile and Yeh \cite{HY} proved
\begin{equation}
\sum_{i=1}^{k} \frac{\Lambda_{i}^{\frac{1}{2}}}{\Lambda_{k+1}-\Lambda_{i}} \geq \frac{n^{2} k^{3 / 2}}{8(n+2)}\left(\sum_{i=1}^{k} \Lambda_{i}\right)^{-\frac{1}{2}},
\end{equation}
by making use of an improved method of Hile and Protter \cite{HP}. Furthermore, in
$1990,$ Hook \cite{Hook} Chen and Qian \cite{CQian} proved, independently, the following inequality:

\begin{equation}\label{CQH-ineq}
\frac{n^{2} k^{2}}{8(n+2)} \leq\left[\sum_{i=1} \frac{\Lambda_{i}^{\frac{1}{2}}}{\Lambda_{k+1}-\Lambda_{i}}\right] \sum_{i=1}^{k} \Lambda_{i}^{\frac{1}{2}}.
\end{equation}
Recently, in \cite{A1}, a survey paper on recent developments of eigenvalue problems, Ashbaugh pointed out whether one can establish inequalities for eigenvalues of the vibrating clamped plate problem which are analogous inequalities of Yang in the case of the eigenvalue problem of the Laplacian with Dirichlet boundary condition. In \cite{CY3}, Cheng and Yang gave an affirmative answer to the problem introduced by Ashbaugh. This is to say that they obtain the following universal inequality:
\begin{equation}\Lambda_{k+1}-\frac{1}{k} \sum_{i=1}^{k} \Lambda_{i} \leq\left[\frac{8(n+2)}{n^{2}}\right]^{\frac{1}{2}} \frac{1}{k} \sum_{i=1}^{k}\left[\Lambda_{i}\left(\Lambda_{k+1}-\Lambda_{i}\right)\right]^{\frac{1}{2}},\end{equation}which is sharper than
\begin{equation}\label{1.11-ine}
\Lambda_{k+1} \leq\left[1+\frac{8(n+2)}{n^{2}}\right] \frac{1}{k} \sum_{i=1}^{k} \Lambda_{i}.
\end{equation}
It is easy to see that inequality \eqref{1.11-ine} is better than inequality \eqref{PPW-ine-2} of Payne, P\'{o}lya and Weinberger. We shall also discuss the relation between inequality \eqref{1.11-ine} and inequality \eqref{CQH-ineq} introduced by Hook \cite{Hook}, and Chen and Qian \cite{CQian} in the Remark 2 of Section 2.

Assume that $X : \mathcal{M}^{n}\rightarrow \mathbb{R}^{n+p}$ is an isometric immersion from an $n$-dimensional, oriented, complete
Riemannian manifold $\mathcal{M}^{n}$ to the Euclidean space $\mathbb{R}^{n+p}$, $\left\{e_{1}, \ldots, e_{n}\right\}$ is a local orthonormal basis of $\mathcal{M}^{n}$ with respect to the induced metric, and $\{e_{n+1}, \ldots, e_{n+p}\}$ is the local unit orthonormal normal vector fields. Assume that $$\textbf{H} =\frac{1}{n}\sum_{\alpha=n+1}^{n+p} H^{\alpha} e_{\alpha}=\frac{1}{n}\sum_{\alpha=n+1}^{n+p}\left(\sum_{i=1}^{n} h_{i i}^{\alpha}\right) e_{\alpha}$$ is the mean curvature vector field, and

$$H=|\textbf{H}|=\frac{1}{n} \sqrt{\sum_{\alpha=n+1}^{n+p}\left(\sum_{i=1}^{n} h_{i i}^{\alpha}\right)^{2}}$$ is the mean curvature of $\mathcal{M}^{n}$ throughout this paper. Let us use $\Pi$ to denote the set of all isometric immersions from $\mathcal{M}^{n}$ into the Euclidean space $\mathbb{R}^{n+p}$.  Wang and Xia \cite{WX3} proved the following inequality:

\begin{equation}\begin{aligned}
\label{WX-ineq}\sum_{i=1}^{k}\left(\Lambda_{k+1}-\Lambda_{i}\right)^{2} \leq & \frac{4}{n}\left\{\sum_{i=1}^{k}\left(\Lambda_{k+1}-\Lambda_{i}\right)^{2}\left[\left( \frac{n}{2}+1\right) \Lambda_{i}^{\frac{1}{2}}+C_{0}\right]\right\}^{\frac{1}{2}} \\
& \times\left\{\sum_{i=1}^{k}\left(\Lambda_{k+1}-\Lambda_{i}\right)\left(\Lambda_{i}^{\frac{1}{2}}+C_{0}\right)\right\}^{\frac{1}{2}},
\end{aligned}\end{equation}where $$C_{0}=\frac{1}{4}\inf _{\sigma \in \Pi}\max_{\Omega}\left(n^{2}H^{2}\right).$$Here, we remark that inequality \eqref{WX-ineq} is not a universal bound since it contains mean curvature terms depending on the bounded domain $\Omega$. However, when $\mathcal{M}^{n}$  is an $n$-dimensional minimal submanifolds on a unit sphere, then we deduce to a universal inequality obtained by Wang and Xia in \cite{WX1}.

Let $X : \mathcal{M}^{n}\rightarrow \mathbb{R}^{n+p}$ be an isometric immersion from an $n$-dimensional, oriented, complete
Riemannian manifold $\mathcal{M}^{n}$ to the Euclidean space $\mathbb{R}^{n+p}$.  We consider a smooth family of immersions
$X_{t} = X(\cdot,t):\mathcal{M}^{n}\rightarrow \mathbb{R}^{n+p}$ with corresponding images $\mathcal{M}^{n}_{t} = X_{t}(\mathcal{M}^{n})$ such that the following mean curvature equation system \cite{H}:

\begin{equation}\label{MCF-Equa}{\begin{cases}
&\frac{d}{dt}X(x,t)=\textbf{H}(x,t), x\in \mathcal{M}^{n},   \\
& X(\cdot,0) = X(\cdot),
\end{cases}}\end{equation}
is satisfied, where $\textbf{H}(x,t)$ is the mean curvature vector of $\mathcal{M}_{t}$ at $X(x, t)$ in $\mathbb{R}^{n+p}$. We let $\nu_{0}$ be a constant vector with is a constant vector with unit length in $\mathbb{R}^{n+p}$.
A submanifold $X:\mathcal{M}^{n}\rightarrow\mathbb{R}^{n+p}$ is said to be a translating soliton of the mean curvature flow \eqref{MCF-Equa}, if it
satisfies

\begin{equation}\label{tran} \textbf{H}=\nu_{0}^{\perp},\end{equation} where $\nu_{0}^{\perp}$ denotes the normal projection
of $\nu_{0}$ to the normal bundle of $\mathcal{M}^{n}$ in $\mathbb{R}^{n+p}$.
Translating solitons are not only special solutions to the mean curvature
flow equations \eqref{MCF-Equa}, but they often occur as Type-II singularity of a mean curvature flow, which plays an important role in the study
of the mean curvature flow \cite{AV}.
In \cite{Xin2}, Xin studied some important properties of translating solitons: the volume growth, generalized maximum principle, Gauss maps and certain functions related to the Gauss maps. Meanwhile, he carried out point-wise estimates and integral estimates for the squared norm of the second fundamental form. According to these estimates,  Xin proved some rigidity theorems for translating solitons in the Euclidean space with higher codimension. In addition, by using a new Omori-Yau maximal principle, Chen and Qiu \cite{ChQ} established a nonexistence theorem for the
translating solitons in the setting of spacelike.  Let $\nu\in\mathbb{R}^{n+p}$ be a constant vector filed defined on the Riemannian manifold $\mathcal{M}^{n}$. Throughout this paper, we use $\langle\cdot,\cdot\rangle_{g}$, $|\cdot|_{g}^{2}$, ${\rm div}$, $\Delta$, $\nabla$ and $\nu^{\top}$ to denote the Riemannian inner product with respect to the induced metric $g$, norm associated with the inner product $\langle\cdot,\cdot\rangle_{g}$,  divergence, Laplacian, the gradient operator on Riemannian manifolds $\mathcal{M}^{n}$ and the projection of the vector $\nu$ on the tangent bundle of $\mathcal{M}^{n}$, respectively. Similarly, we use $\langle\cdot,\cdot\rangle_{g_{0}}$, $|\cdot|_{g_{0}}^{2}$, to denote the Euclidean inner product with respect to the metric $g_{0}$ and the Euclidean norm associated with the inner product $\langle\cdot,\cdot\rangle_{g_{0}}$ on the Euclidean space $\mathbb{R}^{n+p}$, respectively.
Then, we define a new elliptic operator on the Riemannian manifolds $\mathcal{M}^{n}$ as follows:

\begin{equation} \label{L-v} \mathfrak{L}_{\nu}(\cdot) =\Delta(\cdot)+ \langle\nu,\nabla(\cdot)\rangle_{g_{0}}=e^{-\langle\nu,X\rangle_{g_{0}}}{\rm div}(e^{\langle\nu,X\rangle_{g_{0}}}\nabla(\cdot)),\end{equation}which is  introduced by Xin
in \cite{Xin2} and similar to the $\mathfrak{L}$ operator introduced by  Colding and Minicozzi in \cite{CM}.
We remark that the constant vector $\nu$ given by \eqref{L-v} do not necessary satisfies \eqref{tran}. In particular,  $\mathcal{L}_{\nu_{0}}$ operator will be denoted by $\mathcal{L}_{II}$ to agree with the notation appeared in \cite{Xin2}.
 It can be shown that the elliptic differential operator $\mathfrak{L}_{\nu}$ is a self-adjoint operator with
respect to the weighted measure $e^{\langle\nu,X\rangle_{g_{0}}}dv$.

In this paper, we investigate  the following clamped plate problem of $\mathfrak{L}_{\nu}^{2}$ operator on complete Riemannian manifold  $\mathcal{M}^{n}$:

\begin{equation}\label{L-2-prob} {\begin{cases} \
 \mathfrak{L}_{\nu}^{2}u  =\Lambda u , \ \ & {\rm in} \ \ \ \ \Omega, \\
  \ u=\frac{\partial u}{\partial \textbf{n}}=0, \ \ & {\rm on} \ \ \partial \Omega,
\end{cases}}\end{equation}where $\textbf{n}$ denotes the normal vector to the boundary $\partial\Omega$. In \cite{Z}, the author studied the above eigenvalue problem, and esatblished some inequalities of eigenvalues
with lower order.
In general, it is very difficult to establish the universal eigenvalue inequalities of weighted Laplacian on the complete Riemannian manifolds.  Even for Laplacian, some mathematicians obtained some universal inequalities only for minimal submanifolds on the Euclidean spaces, unit spheres, projective spaces, hyperbolic spaces, homogeneous Riemannian manifolds and so on. For example, see \cite{Bran,CZ,CY1,CY3,HP,HY,Hook,LP,PPW2,WX1,WX2,WX3,Z} and references therein. Therefore,
it is natural to ask the following problem.

\begin{prob}Under what condition the eigenvalue inequalities of $\mathcal{L}^{2}_{\nu}$ operator on
the complete Riemannian manifolds  do not
depend on the mean curvature? Furthermore, can we establish a universal inequality
of Yang type for eigenvalue problem \eqref{L-2-prob} under such a condition?\end{prob}
Inspired by previous research works, we continue to study eigenvalues of eigenvalue
problem \eqref{L-2-prob} of $\mathcal{L}^{2}_{\nu}$ operator on a bounded domain in an $n$-dimensional complete Riemannian manifold $\mathcal{M}^{n}$.
However, we only focus on the eigenvalues with higher order in this paper. Assume that $\mathcal{M}^{n}$ is a translating soliton and $\nu$ is unit vector field, we give an affirmative answer to this problem. In order to establish a universal bound for the $\mathcal{L}^{2}_{II}$ operator on the translating soliton, we need to prove the following theorem.
\begin{thm}\label{thm1.1}
Let $(\mathcal{M}^{n},g)$ be an $n$-dimensional complete Riemannian manifold isometrically embedded into the Euclidean space $\mathbb{R}^{n+p}$ with mean curvature $H$, then
eigenvalues $\Lambda
_{i}$ of the clamped plate problem \eqref{L-2-prob} of  the $\mathfrak{L}_{\nu}^{2}$ operator satisfy
\begin{equation}\label{thm-1.1-ineq}\begin{aligned}
\sum_{i=1}^{k}\left(\Lambda_{k+1}-\Lambda_{i}\right)^{2}
\leq & \frac{4}{n}\left\{\sum_{i=1}^{k}\left(\Lambda_{k+1}-\Lambda_{i}\right)^{2}\left(\left(\frac{n}{2}+1\right) \Lambda_{i}^{\frac{1}{2}}+4\widetilde{C}_{1}\Lambda_{1}^{\frac{1}{4}}+4\widetilde{C}_{1}^{2}+C_{1}\right)\right\}^{\frac{1}{2}} \\
& \times\left\{\sum_{i=1}^{k}\left(\Lambda_{k+1}-\Lambda_{i}\right)\left(\Lambda_{i}^{\frac{1}{2}}+4\widetilde{C}_{1}\Lambda_{1}^{\frac{1}{4}}+4\widetilde{C}_{1}^{2}+C_{1}\right)\right\}^{\frac{1}{2}},
\end{aligned}\end{equation}
where $C_{1}$ is a constant given by
$$
C_{1}=\frac{1}{4}\inf _{\sigma \in \Pi}\max_{\Omega}\left(n^{2}H^{2}\right),
$$and $\widetilde{C}_{1}$ is given by $$\widetilde{C}_{1}=\frac{1}{4}\max_{\Omega} |\nu^{\top}|_{g_{0}}.$$
\end{thm}

\begin{rem}In theorem {\rm\ref{thm1.1}}, assuming that $|\nu^{\top}|_{g_{0}}=0$  then the constant is given by
$$
C_{1}=\frac{1}{4} \inf _{\sigma \in \Pi} \max _{\Omega}\left(n^{2}H^{2}\right).
$$Therefore, inequality \eqref{thm-1.1-ineq} covers inequality \eqref{WX-ineq} given by Wang and Xia in {\rm \cite{WX3}}.

\end{rem}

As an application of theorem \ref{thm1.1}, we investigate the eigenvalues of $\mathfrak{L}_{II}^{2}$ operator on the translating solitons and prove the following theorem.

\begin{thm}\label{thm1.2}\textbf{\emph{(Domain Independent Bound)}}
Let $(\mathcal{M}^{n},g)$ be an $n$-dimensional complete translating soliton isometrically embedded into the Euclidean space $\mathbb{R}^{n+p}$, then
eigenvalues $\Lambda
_{i}$ of the clamped plate problem \eqref{L-2-prob} of  the $\mathfrak{L}_{II}^{2}$ operator satisfy
\begin{equation}\label{thm-1.1-ineq-trans}\begin{aligned}
\sum_{i=1}^{k}\left(\Lambda_{k+1}-\Lambda_{i}\right)^{2}
\leq & \frac{4}{n}\left\{\sum_{i=1}^{k}\left(\Lambda_{k+1}-\Lambda_{i}\right)^{2}\left(\left(\frac{n}{2}+1\right) \Lambda_{i}^{\frac{1}{2}} + \Lambda^{\frac{1}{4}}_{i}+\frac{n^{2}}{4}\right)\right\}^{\frac{1}{2}} \\
& \times\left\{\sum_{i=1}^{k}\left(\Lambda_{k+1}-\Lambda_{i}\right)\left(\Lambda_{i}^{\frac{1}{2}} + \Lambda^{\frac{1}{4}}_{i}+\frac{n^{2}}{4}\right)\right\}^{\frac{1}{2}}.
\end{aligned}\end{equation}

\end{thm}

\begin{rem}Clearly,  eigenvalue inequality \eqref{thm-1.1-ineq-trans} is universal. Therefore, our work can be regarded as a new contribution to this research direction.
\end{rem}
This paper is organized as follows. In section \ref{sec2}, we prove a general formula for
eigenvalues of clamped plate problem  \eqref{L-2-prob}. Furthermore, we prove theorem \ref{thm1.1} in section \ref{sec3} by applying general formula and some results of Chen-Cheng type. In addition, we prove theorem \ref{thm1.2} in the remainder of this section.  As several attractive applications of theorem \ref{thm1.1}, we further consider the eigenvalues of $\mathfrak{L}_{\nu}^{2}$ operator of the minimal submanifolds on the Euclidean spaces, submanifolds on the unit spheres and projective spaces in section \ref{sec4}.

\section{Key lemma and its proof} \label{sec2}
\vskip3mm
In this section, we would like to establish a general formula, which will play an important role in the proof  of theorem \ref{thm1.1}.  Our  general formula says the following.

\begin{lem} \label{lemma2.1} Let $\Omega$ be a bounded domain on an $n$-dimensional complete Riemannian manifold $(\mathcal{M}^{n},g)$, and $\Lambda_{i}$ be the $i^{\text {th}}$ eigenvalue of the eigenvalue problem \eqref{L-2-prob} and $u_{i}$ be the orthonormal eigenfunction corresponding to $\Lambda_{i},$ that is,
$$
\left\{\begin{array}{ll}
\mathfrak{L}_{\nu}^{2} u_{i}=\Lambda_{i} u_{i}, & \text { in } \Omega, \\
u_{i}=\frac{\partial u_{i}}{\partial v}=0, & \text { on } \partial \Omega, \\
\int_{\Omega} u_{i} u_{j} e^{\langle\nu,X\rangle_{g_{0}}}dv=\delta_{i j}, & \forall i, j=1,2, \ldots,
\end{array}\right.
$$
where $\nu$ is an outward normal vector field of $\partial \Omega$. Then, for any function $f  \in C^{4}(\Omega) \cap C^{3}(\partial \Omega)$ and any positive integer $k,$ we have

\begin{equation}\begin{aligned}\label{lem-2.1} &
\sum_{i=1}^{k}\left(\Lambda_{k+1}-\Lambda_{i}\right)^{2} \int_{\Omega} u_{i}^{2}|\nabla f |_{g}^{2} e^{\langle\nu,X\rangle_{g_{0}}}dv\\&\quad\quad\leq \sum_{i=1}^{k} \delta\left(\Lambda_{k+1}-\Lambda_{i}\right)^{2}\int_{\Omega}\Psi_{i}(f)e^{\langle\nu,X\rangle_{g_{0}}}dv\\&\quad\quad
+\sum_{i=1}^{k} \frac{\left(\Lambda_{k+1}-\Lambda_{i}\right)}{\delta}\int_{\Omega}\Theta_{i}(f)e^{\langle\nu,X\rangle_{g_{0}}}dv;
\end{aligned}\end{equation}
where

\begin{equation}\label{Psi}\Psi_{i}(f )=-2|\nabla f |_{g}^{2} u_{i} \mathfrak{L}_{\nu} u_{i}+4 u_{i} \mathfrak{L}_{\nu} f \left\langle\nabla f , \nabla u_{i}\right\rangle_{g}+4\left\langle\nabla f , \nabla u_{i}\right\rangle_{g}^{2}+u_{i}^{2}\left( \mathfrak{L}_{\nu} f \right)^{2},\end{equation}

\begin{equation}\label{Theta}\Theta_{i}(f)=\left(\left\langle\nabla f , \nabla u_{i}\right\rangle_{g}+\frac{u_{i} \mathfrak{L}_{\nu} f }{2}\right)^{2} ,\end{equation}
and
$\delta$ is any positive constant.
\end{lem}

\begin{proof}
Let $$\varphi_{i}=f  u_{i}-\sum_{j=1}^{k} a_{i j} u_{j},$$ for any integer $k \geq 1,$ where
$$
a_{i j}=\int_{\Omega} f  u_{i} u_{j} e^{\langle\nu,X\rangle_{g_{0}}}dv=a_{j i},
$$then we have
$$
\left.\varphi_{i}\right|_{\partial \mathcal{M}^{n}}=\left.\frac{\partial \varphi_{i}}{\partial \nu}\right|_{\partial \mathcal{M}^{n}}=0,$$and

\begin{equation}\label{int-om-vu}\int_{\Omega} \varphi_{i} u_{j} e^{\langle\nu,X\rangle_{g_{0}}}dv=0, \quad \forall i, j=1, \ldots, k.
\end{equation}
From the Rayleigh-Ritz inequality, we get

\begin{equation}\label{RR}\Lambda_{k+1} \int_{\Omega} \varphi_{i}^{2} e^{\langle\nu,X\rangle_{g_{0}}}dv \leq \int_{\Omega} \varphi_{i} \mathfrak{L}_{\nu}^{2} \varphi_{i} e^{\langle\nu,X\rangle_{g_{0}}}dv.\end{equation}By direct computation, we have

\begin{equation*}
\mathfrak{L}_{\nu}\left(f
 u_{i}\right)=f  \mathfrak{L}_{\nu} u_{i}+2\left\langle\nabla f , \nabla u_{i}\right\rangle_{g}+u_{i} \mathfrak{L}_{\nu} f ,
\end{equation*}
and

\begin{equation}
\begin{aligned}\label{L-2-fui}
&\mathfrak{L}_{\nu}^{2}\left(f  u_{i}\right)\\&= \mathfrak{L}_{\nu}\left(f  \mathfrak{L}_{\nu} u_{i}+2\left\langle\nabla f , \nabla u_{i}\right\rangle_{g}+u_{i} \mathfrak{L}_{\nu} f \right) \\
&= f  \mathfrak{L}_{\nu}^{2} u_{i}+2\left\langle\nabla f , \nabla\left( \mathfrak{L}_{\nu} u_{i}\right)\right\rangle_{g}+\mathfrak{L}_{\nu} f  \mathfrak{L}_{\nu} u_{i} +2 \mathfrak{L}_{\nu}\left(\left\langle\nabla f , \nabla u_{i}\right\rangle_{g}\right)+ \mathfrak{L}_{\nu}\left(u_{i} \mathfrak{L}_{\nu} f \right) \\
&= \Lambda_{i} f  u_{i}+\omega_{i},
\end{aligned}
\end{equation}
where

\begin{equation*}
\omega_{i}=2\left\langle\nabla f , \nabla\left( \mathfrak{L}_{\nu} u_{i}\right)\right\rangle_{g}+ \mathfrak{L}_{\nu} f  \mathfrak{L}_{\nu} u_{i}+2 \mathfrak{L}_{\nu}\left(\left\langle\nabla f , \nabla u_{i}\right\rangle_{g}\right)+ \mathfrak{L}_{\nu}\left(u_{i} \mathfrak{L}_{\nu} f \right).
\end{equation*}
From \eqref{L-2-fui} and \eqref{int-om-vu}, we can get

\begin{equation}
\begin{aligned}\label{v-i-L}
\int_{\Omega} \varphi_{i} \mathfrak{L}_{\nu}^{2} \varphi_{i} e^{\langle\nu,X\rangle_{g_{0}}}dv &=\int_{\Omega} \varphi_{i} \mathfrak{L}_{\nu}^{2}\left(f  u_{i}\right) e^{\langle\nu,X\rangle_{g_{0}}}dv \\
&=\Lambda_{i} \int_{\Omega} \varphi_{i} f  u_{i} e^{\langle\nu,X\rangle_{g_{0}}}dv+\int_{\Omega} \varphi_{i} \omega_{i} e^{\langle\nu,X\rangle_{g_{0}}}dv \\
&=\Lambda_{i} \int_{\Omega} \varphi_{i}^{2} e^{\langle\nu,X\rangle_{g_{0}}}dv+\int_{\Omega} f  u_{i} \omega_{i} e^{\langle\nu,X\rangle_{g_{0}}}dv-\sum_{j=1}^{k} a_{i j} b_{i j},
\end{aligned}
\end{equation}
where

\begin{equation}\label{bij}b_{i j}=\int_{\Omega} \omega_{i} u_{j} e^{\langle\nu,X\rangle_{g_{0}}}dv.\end{equation}
By \eqref{RR}, \eqref{v-i-L} and \eqref{bij}, we have

\begin{equation}
\begin{aligned}\label{int-nf-nu}
\left(\Lambda_{k+1}-\Lambda_{i}\right) \int_{\Omega} \varphi_{i}^{2} e^{\langle\nu,X\rangle_{g_{0}}}dv \leq \int_{\Omega} f u_{i} \omega_{i} e^{\langle\nu,X\rangle_{g_{0}}}dv-\sum_{j=1}^{k} a_{i j} b_{i j}.
\end{aligned}
\end{equation}
Using integration by parts, we have

\begin{equation}
\begin{aligned}\label{lf-luj}
&\int\left\langle\nabla f , \nabla u_{j}\right\rangle_{g} \mathfrak{L}_{\nu} u_{i} e^{\langle\nu,X\rangle_{g_{0}}}dv\\
&=-\int_{\Omega} u_{j} {\rm div} \left(e^{\langle\nu,X\rangle_{g_{0}}} \mathfrak{L}_{\nu} u_{i} \nabla f \right) dv \\
&=-\int_{\Omega} u_{j} e^{\langle\nu,X\rangle_{g_{0}}}\left(\left\langle\nabla f , \nabla\left( \mathfrak{L}_{\nu} u_{i}\right)\right\rangle_{g}+ \mathfrak{L}_{\nu} u_{i} \Delta f + \mathfrak{L}_{\nu} u_{i}\langle\nu, \nabla f \rangle_{g_{0}}\right) dv \\
&=-\int_{\Omega} u_{j}\left(\left\langle\nabla f , \nabla\left( \mathfrak{L}_{\nu} u_{i}\right)\right\rangle_{g}+ \mathfrak{L}_{\nu} u_{i} \mathfrak{L}_{\nu} f \right) e^{\langle\nu,X\rangle_{g_{0}}}dv.
\end{aligned}
\end{equation}
 From \eqref{lf-luj}, we infer that,

\begin{equation}\begin{aligned}\label{l-i-l-j}
&\left(\Lambda_{j}-\Lambda_{i}\right) a_{i j} \\&
=\int_{\Omega}\left(f u_{i} \mathfrak{L}_{\nu}^{2} u_{j}-f  u_{j} \mathfrak{L}_{\nu}^{2} u_{i}\right) e^{\langle\nu,X\rangle_{g_{0}}}dv \\&
=\int_{\Omega}\left( \mathfrak{L}_{\nu}\left(f  u_{i}\right) \mathfrak{L}_{\nu} u_{j}- \mathfrak{L}_{\nu}\left(f  u_{j}\right) \mathfrak{L}_{\nu} u_{i}\right) e^{\langle\nu,X\rangle_{g_{0}}}dv \\&
=\int_{\Omega}\left[\left(u_{i} \mathfrak{L}_{\nu} f +2\left\langle\nabla f , \nabla u_{i}\right\rangle_{g}\right) \mathfrak{L}_{\nu} u_{j}\right.\\&\quad-\left.\left(u_{j} \mathfrak{L}_{\nu} f +2\left\langle\nabla f , \nabla u_{j}\right\rangle_{g}\right) \mathfrak{L}_{\nu} u_{i}\right] e^{\langle\nu,X\rangle_{g_{0}}}dv \\&
=\int_{\Omega}\Big{[}u_{j}\left( \mathfrak{L}_{\nu}\left(u_{i} \mathfrak{L}_{\nu} f \right)+2 \mathfrak{L}_{\nu}\left\langle\nabla f , \nabla u_{i}\right\rangle_{g}\right)\\&
\quad\quad\quad\quad\quad-u_{j} \mathfrak{L}_{\nu} f  \mathfrak{L}_{\nu} u_{i}+2 u_{j}\left(\left\langle\nabla f , \nabla\left( \mathfrak{L}_{\nu} u_{i}\right)\right\rangle_{g}+ \mathfrak{L}_{\nu} f  \mathfrak{L}_{\nu} u_{i}\right)\Big{]} e^{\langle\nu,X\rangle_{g_{0}}}dv \\&
=\int_{\Omega} u_{j}\left[ \mathfrak{L}_{\nu}\left(u_{i} \mathfrak{L}_{\nu} f \right)+2 \mathfrak{L}_{\nu}\left(\left\langle\nabla f , \nabla u_{i}\right\rangle_{g}\right)\right.\\&\quad\left.+2\left\langle\nabla f , \nabla\left( \mathfrak{L}_{\nu} u_{i}\right)\right\rangle_{g}+ \mathfrak{L}_{\nu} f  \mathfrak{L}_{\nu} u_{i}\right) e^{\langle\nu,X\rangle_{g_{0}}}dv \\&
=\int_{\Omega} \omega_{i} u_{j} e^{\langle\nu,X\rangle_{g_{0}}}dv\\
&=b_{i j}.
\end{aligned}\end{equation}It follows from \eqref{int-nf-nu} and \eqref{l-i-l-j} that

\begin{equation*}\begin{aligned}
\left(\Lambda_{k+1}-\Lambda_{i}\right) \int_{\Omega} \varphi_{i}^{2} e^{\langle\nu,X\rangle_{g_{0}}}dv &\leq \int_{\Omega} \varphi_{i} \omega_{i} e^{\langle\nu,X\rangle_{g_{0}}}dv\\&=\int_{\Omega} f  u_{i} \omega_{i} e^{\langle\nu,X\rangle_{g_{0}}}dv-\sum_{j=1}^{k}\left(\Lambda_{j}-\Lambda_{i}\right) a_{i j}^{2}.
\end{aligned}\end{equation*}
On the other hand, from \eqref{int-om-vu} and \eqref{bij}, we have

\begin{equation}
\begin{aligned}\label{lk+1-li}
&\left(\Lambda_{k+1}-\Lambda_{i}\right)\left(\int_{\Omega} \varphi_{i} \omega_{i} e^{\langle\nu,X\rangle_{g_{0}}}dv\right)^{2} \\ &=\left(\Lambda_{k+1}-\Lambda_{i}\right)\left(\int_{\Omega} \varphi_{i}\left(\omega_{i}-\sum_{j=1}^{k} b_{i j} u_{j}\right) e^{\langle\nu,X\rangle_{g_{0}}}dv\right)^{2} \\
& \leq\left(\Lambda_{k+1}-\Lambda_{i}\right)\left\|\varphi_{i}\right\|^{2}\left(\left\|\omega_{i}\right\|^{2}-\sum_{j=1}^{k} b_{i j}^{2}\right) \\
& \leq \int_{\Omega} \varphi_{i} \omega_{i} e^{\langle\nu,X\rangle_{g_{0}}}dv\left(\left\|\omega_{i}\right\|^{2}-\sum_{j=1}^{k} b_{i j}^{2}\right),
\end{aligned}
\end{equation}
where $$\left\|\varphi_{i}\right\|^{2}=\int_{\Omega} \varphi_{i}^{2} e^{\langle\nu,X\rangle_{g_{0}}}dv .$$ By Bessel inequality and \eqref{bij}, we know that

\begin{equation}\label{omega}\left\|\omega_{i}\right\|^{2}-\sum_{j=1}^{k} b_{i j}^{2}\geq0.\end{equation}It follows from \eqref{lk+1-li} and \eqref{omega} that

\begin{equation}\label{Bel}\int_{\Omega} \varphi_{i} \omega_{i} e^{\langle\nu,X\rangle_{g_{0}}}dv\geq0.\end{equation}
From \eqref{lk+1-li} and \eqref{Bel}, we can get

\begin{equation}\label{lk+1-li-2}
\left(\Lambda
_{k+1}-\Lambda_{i}\right) \int_{\Omega} \varphi_{i} \omega_{i} e^{\langle\nu,X\rangle_{g_{0}}}dv \leq\left\|\omega_{i}\right\|^{2}-\sum_{j=1}^{k} b_{i j}^{2}.
\end{equation}
Multiplying \eqref{lk+1-li-2} by ( $\Lambda_{k+1}-\Lambda_{i}$ ) and summing on $i$ from 1 to $k$, and using \eqref{l-i-l-j}, we get

\begin{equation}\label{sum-1}\begin{aligned}
\sum_{i=1}^{k}\left(\Lambda_{k+1}-\Lambda_{i}\right)^{2} \int_{\Omega} \varphi_{i} \omega_{i}e^{\langle\nu,X\rangle_{g_{0}}}dv \leq & \sum_{i=1}^{k}\left(\Lambda_{k+1}-\Lambda_{i}\right)\left\|\omega_{i}\right\|^{2}\\&-\sum_{i, j=1}^{k}\left(\Lambda_{k+1}-\Lambda_{i}\right)\left(\Lambda_{i}-\Lambda_{j}\right)^{2} a_{i j}^{2}.\end{aligned}
\end{equation}
Using integration by parts, we have
$$
\begin{aligned}
&\int_{\Omega} u_{j}\left\langle\nabla f , \nabla u_{i}\right\rangle_{g} e^{\langle\nu,X\rangle_{g_{0}}}dv\\
&=-\int_{\Omega} u_{i} \operatorname{div}\left(u_{j} e^{\langle\nu,X\rangle_{g_{0}}} \nabla f \right) dv \\
&=-\int_{\Omega} u_{i} e^{\langle\nu,X\rangle_{g_{0}}}\left(\left\langle\nabla f , \nabla u_{j}\right\rangle_{g}+u_{j}\langle\nu, \nabla f \rangle+u_{j} \Delta f \right) dv \\
&=-\int_{\Omega} u_{i}\left(\left\langle\nabla f , \nabla u_{j}\right\rangle_{g}+u_{j} \mathfrak{L}_{\nu} f \right) e^{\langle\nu,X\rangle_{g_{0}}}
dv,
\end{aligned}
$$and
$$
\begin{aligned}&
\int_{\Omega} f  u_{i}\left\langle\nabla f , \nabla u_{i}\right\rangle_{g} e^{\langle\nu,X\rangle_{g_{0}}}dv \\&
\quad=\int_{\Omega} f  u_{i} e^{\langle\nu,X\rangle_{g_{0}}}
\left\langle\nabla f , \nabla u_{i}\right\rangle_{g} dv \\&
\quad=-\int_{\Omega} u_{i} \operatorname{div}\left(f  u_{i} e^{\langle\nu,X\rangle_{g_{0}}} \nabla f \right) dv \\&
\quad=-\int_{\Omega} u_{i} e^{\langle\nu,X\rangle_{g_{0}}}\left(u_{i}|\nabla f |_{g}^{2}+f \left\langle\nabla f , \nabla u_{i}\right\rangle_{g}+f  u_{i}\langle\nu, \nabla f\rangle+f  u_{i} \Delta f \right) dv \\&
\quad=-\int_{\Omega}\left(u_{i}^{2}|\nabla f |_{g}^{2}+f  u_{i}^{2} \mathfrak{L}_{\nu} f \right) e^{\langle\nu,X\rangle_{g_{0}}}dv-\int_{\Omega} f  u_{i}\left\langle\nabla f, \nabla u_{i}\right\rangle_{g} e^{\langle\nu,X\rangle_{g_{0}}}dv,
\end{aligned}
$$
which implies that
$$
2 \int_{\Omega} f  u_{i}\left\langle\nabla f , \nabla u_{i}\right\rangle_{g} e^{\langle\nu,X\rangle_{g_{0}}}dv=-\int_{\Omega}\left(u_{i}^{2}|\nabla f |_{g}^{2}+f  u_{i}^{2} \mathfrak{L}_{\nu} f \right) e^{\langle\nu,X\rangle_{g_{0}}}dv.
$$
Setting $$t_{i j}=\int_{\Omega} u_{j}\left(\left\langle\nabla f , \nabla u_{i}\right\rangle_{g}+\frac{u_{i} \mathfrak{L}_{\nu} f }{2}\right) e^{\langle\nu,X\rangle_{g_{0}}}dv,$$ we have $$t_{ij}=-t_{ji},$$ and

\begin{equation}\label{int-2}
\begin{aligned}&
\int_{\Omega}-2 \varphi_{i}\left(\left\langle\nabla f , \nabla u_{i}\right\rangle_{g}+\frac{u_{i} \mathfrak{L}_{\nu} f }{2}\right) e^{\langle\nu,X\rangle_{g_{0}}}dv \\&
\quad=\int_{\Omega}\left(-2 f  u_{i}\left\langle\nabla f , \nabla u_{i}\right\rangle_{g}-f  u_{i}^{2} \mathfrak{L}_{\nu} f \right) e^{\langle\nu,X\rangle_{g_{0}}}dv+2 \sum_{j=1}^{k} a_{i j} t_{i j} \\&
\quad=\int_{\Omega} u_{i}^{2}|\nabla f |_{g}^{2} e^{\langle\nu,X\rangle_{g_{0}}}dv+2 \sum_{j=1}^{k} a_{ij} t_{ij}.
\end{aligned}
\end{equation}
By \eqref{sum-1}, \eqref{int-2} and the Schwarz inequality, we get

\begin{equation}\begin{aligned}\label{sum-2}&
\left(\Lambda_{k+1}-\Lambda_{i}\right)^{2}\left(\int_{\Omega} u_{i}^{2}|\nabla f |_{g}^{2} e^{\langle\nu,X\rangle_{g_{0}}}dv+2 \sum_{j=1}^{k} a_{i j} t_{i j}\right) \\&
\quad=\left(\Lambda_{k+1}-\Lambda_{i}\right)^{2} \int_{\Omega}-2 \varphi_{i}\left(\left\langle\nabla f, \nabla u_{i}\right\rangle_{g}+\frac{u_{i} \mathfrak{L}_{\nu} f }{2}-\sum_{j=1}^{k} t_{i j} u_{j}\right) e^{\langle\nu,X\rangle_{g_{0}}}dv \\&
\quad \leq \delta\left(\Lambda_{k+1}-\Lambda_{i}\right)^{3} \int_{\Omega} \varphi_{i}^{2} e^{\langle\nu,X\rangle_{g_{0}}}dv\\&
\quad+\frac{\Lambda_{k+1}-\Lambda_{i}}{\delta} \int_{\Omega}\left(\left\langle\nabla f , \nabla u_{i}\right\rangle_{g}+\frac{u_{i} \mathfrak{L}_{\nu} f}{2}-\sum_{j=1}^{k} t_{i j} u_{j}\right)^{2} e^{\langle\nu,X\rangle_{g_{0}}}dv \\&
\quad \leq \delta\left(\Lambda_{k+1}-\Lambda_{i}\right)^{2}\left(\int_{\Omega} f  u_{i} \omega_{i} e^{\langle\nu,X\rangle_{g_{0}}}dv
-\sum_{j=1}^{k}\left(\Lambda_{j}-\Lambda_{i}\right) a_{i j}^{2}\right) \\&
\quad+\frac{\Lambda_{k+1}-\Lambda_{i}}{\delta}\left(\int_{\Omega}\left(\left\langle\nabla f , \nabla u_{i}\right\rangle_{g}
+\frac{u_{i} \mathfrak{L}_{\nu} f }{2}\right)^{2} e^{\langle\nu,X\rangle_{g_{0}}}dv-\sum_{j=1}^{k} t_{i j}^{2}\right),
\end{aligned}\end{equation}
where $\delta$ is any positive constant.  Noticing $a_{i j}=a_{j i}$ and  $t_{i j}=-t_{j i}$, we have
\begin{equation}\begin{aligned}\label{SUM-Labda}&
2 \sum_{i=1}^{k}\left[\left(\Lambda_{k+1}-\Lambda_{i}\right)^{2} \sum_{j=1}^{k}a_{ij}t_{ij}\right]\\&=2 \sum_{i, j=1}^{k}\left(\Lambda_{k+1}-\Lambda_{i}\right)^{2}a_{ij}t_{ij} \\&
=\sum_{i, j=1}^{k}\left(\Lambda_{k+1}-\Lambda_{i}\right)^{2} a_{ij}t_{ij}-\sum_{i, j=1}^{k}\left(\Lambda_{k+1}-\Lambda_{j}\right)^{2} a_{ij}t_{ij} \\&
=-\sum_{i, j=1}^{k}\left(\Lambda_{k+1}-\Lambda_{i}+\Lambda_{k+1}-\Lambda_{j}\right)\left(\Lambda_{i}-\Lambda_{j}\right) a_{ij}t_{ij} \\&
=-2 \sum_{i, j=1}^{k}\left(\Lambda_{k+1}-\Lambda_{i}\right)\left(\Lambda_{i}-\Lambda_{j}\right) a_{ij}t_{ij}.
\end{aligned}\end{equation}
Summing over $i$ from 1 to $k$ in \eqref{sum-2} and utilizing \eqref{SUM-Labda}, we yield

\begin{equation*}\begin{aligned}&
\sum_{i=1}^{k}\left(\Lambda_{k+1}-\Lambda_{i}\right)^{2} \int_{\Omega} u_{i}^{2}|\nabla f
|^{2} e^{\langle\nu,X\rangle_{g_{0}}}dv-2 \sum_{i, j=1}^{k}\left(\Lambda_{k+1}-\Lambda_{i}\right)\left(\Lambda_{i}-\Lambda_{j}\right) a_{i j} t_{i j} \\&
\quad \leq \sum_{i=1}^{k} \delta\left(\Lambda_{k+1}-\Lambda_{i}\right)^{2} \int_{\Omega} f  u_{i} \omega_{i} e^{\langle\nu,X\rangle_{g_{0}}}dv\\&
\quad+\sum_{i=1}^{k} \frac{\Lambda_{k+1}-\Lambda_{i}}{\delta} \int_{\Omega}\left(\left\langle\nabla f , \nabla u_{i}\right\rangle_{g}+\frac{u_{i}\mathfrak{L}_{\nu} f }{2}\right)^{2} e^{\langle\nu,X\rangle_{g_{0}}}dv \\&
\quad-\sum_{i, j=1}^{k} \delta\left(\Lambda_{k+1}-\Lambda_{i}\right)\left(\Lambda_{i}-\Lambda_{j}\right)^{2} a_{i j}^{2}-\sum_{i, j=1}^{k} \frac{\Lambda_{k+1}-\Lambda_{i}}{\delta} t_{i j}^{2}.
\end{aligned}
\end{equation*}
Hence, we have

\begin{equation}\begin{aligned}\label{sum-3}
&\sum_{i=1}^{k}\left(\Lambda_{k+1}-\Lambda_{i}\right)^{2}\int_{\Omega} u_{i}^{2}|\nabla f |_{g}^{2} e^{\langle\nu,X\rangle_{g_{0}}}dv\\&\leq \sum_{i=1}^{k} \delta\left(\Lambda_{k+1}-\Lambda_{i}\right)^{2} \int_{\Omega} f  u_{i} \omega_{i} e^{\langle\nu,X\rangle_{g_{0}}}dv \\&
  +\sum_{i=1}^{k} \frac{\Lambda_{k+1}-\Lambda_{i}}{\delta} \int_{\Omega}\left(\left\langle\nabla f , \nabla u_{i}\right\rangle_{g}+\frac{u_{i} \mathcal{L}_{\nu} f }{2}\right)^{2} e^{\langle\nu,X\rangle_{g_{0}}}dv.
\end{aligned}
\end{equation}
By direct computation, we have

\begin{equation*}\begin{aligned}&
\int_{\Omega} f  u_{i}\left\langle\nabla f , \nabla\left( \mathfrak{L}_{\nu} u_{i}\right)\right\rangle_{g} e^{\langle\nu,X\rangle_{g_{0}}}dv\\&
=-\int_{\Omega} \mathfrak{L}_{\nu} u_{i} div \left(f  u_{i} e^{\langle\nu,X\rangle_{g_{0}}} \nabla f \right) dv \\&
=-\int_{\Omega} \mathfrak{L}_{\nu} u_{i} e^{\langle\nu,X\rangle_{g_{0}}}\left(u_{i}|\nabla f |_{g}^{2}+f \left\langle\nabla f , \nabla u_{i}\right\rangle_{g}+f  u_{i}\langle\nu, \nabla f \rangle+f
 u_{i} \Delta f \right) dv \\&
=-\int_{\Omega}|\nabla f |_{g}^{2} u_{i} \mathfrak{L}_{\nu} u_{i} e^{\langle\nu,X\rangle_{g_{0}}}dv-\int_{\Omega} f  \mathfrak{L}_{\nu} u_{i}\left\langle\nabla f , \nabla u_{i}\right\rangle_{g} e^{\langle\nu,X\rangle_{g_{0}}}dv-\int_{\Omega} f  u_{i} \mathfrak{L}_{\nu} f  \mathfrak{L}_{\nu} u_{i} e^{\langle\nu,X\rangle_{g_{0}}}dv
\end{aligned}
\end{equation*}
From the above equality and the definition of $\omega_{i},$ we have

\begin{equation}\begin{aligned}\label{sum-4}
&\int_{\Omega} f  u_{i} \omega_{i} e^{\langle\nu,X\rangle_{g_{0}}}dv\\
&=\int_{\Omega} f  u_{i}\Bigg\{2\left\langle\nabla f , \nabla\left( \mathfrak{L}_{\nu} u_{i}\right)\right\rangle_{g}+ \mathfrak{L}_{\nu} f
 \mathfrak{L}_{\nu} u_{i}+2 \mathfrak{L}_{\nu}\left(\left\langle\nabla f , \nabla u_{i}\right\rangle_{g}\right)+ \mathfrak{L}_{\nu}
\left(u_{i} \mathfrak{L}_{\nu} f \right)\Bigg\} e^{\langle\nu,X\rangle_{g_{0}}}dv \\&
=-2 \int_{\Omega}|\nabla f |_{g}^{2} u_{i} \mathfrak{L}_{\nu} u_{i} e^{\langle\nu,X\rangle_{g_{0}}}dv-2 \int_{\Omega} f  \mathfrak{L}_{\nu} u_{i}\left\langle\nabla f , \nabla u_{i}\right\rangle_{g} e^{\langle\nu,X\rangle_{g_{0}}}dv\\&\quad-2 \int_{\Omega} f  u_{i} \mathfrak{L}_{\nu} f  \mathfrak{L}_{\nu} u_{i} e^{\langle\nu,X\rangle_{g_{0}}}dv
+\int_{\Omega} f  u_{i} \mathfrak{L}_{\nu} f  \mathfrak{L}_{\nu} u_{i} e^{\langle\nu,X\rangle_{g_{0}}}dv\\&\quad+2 \int u_{i} \mathfrak{L}_{\nu} f \left\langle\nabla f , \nabla u_{i}\right\rangle
_{g}
 e^{\langle\nu,X\rangle_{g_{0}}}dv+4 \int\left\langle\nabla f , \nabla u_{i}\right\rangle_{g}^{2} e^{\langle\nu,X\rangle_{g_{0}}}dv \\&
\quad+2 \int_{\Omega} f  \mathfrak{L}_{\nu} u_{i}\left\langle\nabla f , \nabla u_{i}\right\rangle_{g} e^{\langle\nu,X\rangle_{g_{0}}}dv+\int u_{i}^{2}\left(\mathfrak{L}_{\nu} f \right)^{2} e^{\langle\nu,X\rangle_{g_{0}}}dv\\&
\quad+\int_{\Omega} f u_{i} \mathfrak{L}_{\nu} f  \mathfrak{L}_{\nu} u_{i} e^{\langle\nu,X\rangle_{g_{0}}}dv+2 \int_{\Omega} u_{i} \mathfrak{L}_{\nu} f \left\langle\nabla f , \nabla u_{i}\right\rangle_{g} e^{\langle\nu,X\rangle_{g_{0}}}dv \\&
=\int_{\Omega}\left(-2|\nabla f |_{g}^{2} u_{i} \mathfrak{L}_{\nu} u_{i}+4 u_{i} \mathfrak{L}_{\nu} f \left\langle\nabla f , \nabla u_{i}\right\rangle_{g}+4\left\langle\nabla f , \nabla u_{i}\right\rangle_{g}^{2}+u_{i}^{2}\left( \mathfrak{L}_{\nu} f
\right)^{2}\right) e^{\langle\nu,X\rangle_{g_{0}}}dv.
\end{aligned}\end{equation}
Introducing \eqref{sum-4} into \eqref{sum-3}, we get \eqref{lem-2.1}.
This completes the proof of lemma \ref{lemma2.1}.
\end{proof}

\section{Proof of Main Results}\label{sec3}

\vskip3mm \noindent In this section, we shall give the proofs of the main results. Throughout this paper, we will agree the following convention on ranges of indices:
\[
1 \leq i, j, \cdots, \leq n ; \quad 1 \leq \alpha, \beta, \cdots, \leq n+p.
\]
Suppose that $\left(\overline{x}^{1}, \cdots, \overline{x}^{n}\right)$ is an arbitrary coordinate system in a neighborhood $U$ of $P$ in $\mathcal{M}^{n}$.
Assume that $x$ with components $x^{\alpha}$ defined by
$
x^{\alpha}=x^{\alpha}\left(\overline{x}^{1}, \cdots, \overline{x}^{n}\right),  1 \leq \alpha \leq n+p,
$
is the position vector of $P$ in $\mathbb{R} ^{n+p}$.

\begin{lem}\label{lem3.1}
For  an $n$-dimensional  submanifold  $\mathcal{M}^{n}$ in Euclidean space
$\mathbb{R}^{n+p}$,   let $x=(x^{1},x^{2},\cdots,x^{n+p})$ is the
position vector of a point $p\in \mathcal{M}^{n}$ with
$x^{\alpha}=x^{\alpha}(\overline{x}_{1}, \cdots, \overline{x}_{n})$, $1\leq \alpha\leq
n+p$, where $(x_{1}, \cdots, x_{n})$ denotes a local coordinate
system of $\mathcal{M}^n$. Then, we have
\begin{equation}\label{n-ine}
\sum^{n+p}_{\alpha=1}\langle\nabla x^{\alpha},\nabla x^{\alpha}\rangle_{g}= n,
\end{equation}
\begin{equation}
\begin{aligned}\label{uw-ine}
\sum^{n+p}_{\alpha=1}\langle\nabla x^{\alpha},\nabla u\rangle_{g}\langle\nabla
x^{\alpha},\nabla w\rangle_{g}=\langle\nabla u,\nabla w\rangle_{g},
\end{aligned}
\end{equation}
for any functions  $u, w\in C^{1}(\mathcal{M}^{n})$,
\begin{equation}
\begin{aligned}\label{nH-ine}
\sum^{n+p}_{\alpha=1}(\Delta x^{\alpha})^{2}=n^{2}H^{2},
\end{aligned}
\end{equation}
\begin{equation}
\begin{aligned}\label{0-ine}
\sum^{n+p}_{\alpha=1}\Delta x^{\alpha}\nabla x^{\alpha}= \textbf{0},
\end{aligned}
\end{equation}
where $H$ is the mean curvature of $\mathcal{M}^{n}$.
\end{lem}

From \eqref{n-ine}, we have
\begin{equation}
\int_{\Omega} u_{i}^{2} \sum_{\alpha=1}^{n+p}\left|\nabla x_{\alpha}\right|_{g}^{2} e^{\langle\nu,X\rangle_{g_{0}}}dv=n.
\end{equation}
According to \eqref{uw-ine}, one has

\begin{equation}\label{na-ui-2}
\sum_{\alpha=1}^{n+p}\left\langle\nabla x_{\alpha}, \nabla u_{i}\right\rangle_{g}^{2}=\left|\nabla u_{i}\right|_{g}^{2}.\end{equation}
By a direct calculation, we can conclude that,
\begin{equation}\label{v-2}
\sum_{\alpha=1}^{n+p}\left\langle\nabla x_{\alpha}, \nu\right\rangle_{g_{0}}^{2}=|\nu^{\top}|_{g_{0}}^{2}.\end{equation}
Applying Cauchy-Schwarz inequality and \eqref{v-2}, we derive

\begin{equation}\label{uv-ineq}
\sum_{\alpha=1}^{n+p}\left\langle\nabla x_{\alpha}, \nabla u_{i}\right\rangle_{g}\left\langle\nabla x_{\alpha}, \nu\right\rangle_{g_{0}}\leq\left|\nabla u_{i}\right|_{g}|\nu^{\top} |_{g_{0}}.\end{equation}
It follows from \eqref{0-ine} that,
\begin{equation}\label{xx0-ineq}
\sum_{\alpha=1}^{n+p} \Delta x_{\alpha}\left\langle\nabla x_{\alpha}, \nabla u_{i}\right\rangle_{g} =\sum_{\alpha=1}^{n+p}\left\langle\Delta x_{\alpha}\nabla x_{\alpha}, \nabla u_{i}\right\rangle_{g}=0,
\end{equation}
and

\begin{equation}\label{xx-v}
\sum_{\alpha=1}^{n+p} \Delta x_{\alpha}\left\langle\nabla x_{\alpha}, \nu\right\rangle_{g_{0}}=\sum_{\alpha=1}^{n+p}\left\langle\Delta x_{\alpha}\nabla x_{\alpha}, \nu\right\rangle_{g_{0}}=0.
\end{equation}
From \eqref{uv-ineq} and  \eqref{xx0-ineq}, we have

\begin{equation}\label{vxx-u}
\sum_{\alpha=1}^{n+p} \mathfrak{L}_{\nu} x_{\alpha}\left\langle\nabla x_{\alpha}, \nabla u_{i}\right\rangle_{g}=\sum_{\alpha=1}^{n+p}\left(\Delta x_{\alpha}+\left\langle\nabla x_{\alpha}, \nu\right\rangle_{g_{0}}\right)\left\langle\nabla x_{\alpha}, \nabla u_{i}\right\rangle_{g}\leq\left|\nabla u_{i}\right|_{g}|\nu^{\top}|_{g_{0}}
.
\end{equation}

 \begin{lem}\label{lem-Psi} Let $x_{1}, x_{2}, \ldots, x_{n+p}$ be the standard coordinate functions of $\mathbb{R}^{n+p}$. For any $i=1,2,\cdots k$ and $\alpha=1,2,\cdots,n+p$, let

$$\widehat{\Psi}_{i,\alpha}:=\int_{\Omega}\Psi_{i}(x_{\alpha}) e^{\langle\nu,X\rangle_{g_{0}}}dv,$$
where function $\Psi_{i}$ is given by \eqref{Psi}. Then, we have

\begin{equation}\label{sum-W-Psi}
\begin{aligned}\sum^{n+p}_{\alpha=1}\widehat{\Psi}_{i,\alpha}&\leq\int_{\Omega}\left[-2 n u_{i} \mathfrak{L}_{\nu} u_{i} +4\left|\nabla u_{i}\right|_{g}^{2}+u_{i}^{2}\left(n^{2}H^{2}+|\nu^{\top}|_{g_{0}}^{2}\right)\right] e^{\langle\nu,X\rangle_{g_{0}}}dv\\&\ \ \ \ +4 \Lambda_{i}^{\frac{1}{4}}\left(\int_{\Omega}u_{i}^{2}|\nu^{\top}|^{2}_{g_{0}} e^{\langle\nu,X\rangle_{g_{0}}}dv\right)^{\frac{1}{2}} .
\end{aligned}\end{equation}

\end{lem}
\begin{proof}Taking $f=x_{\alpha}$ in \eqref{Psi} and summing over $\alpha$ from 1 to $n+p,$ we get

\begin{equation}\label{sum-Psi-1}
\begin{aligned}\sum^{n+p}_{\alpha=1}\widehat{\Psi}_{i,\alpha}&=\sum^{n+p}_{\alpha=1}\int_{\Omega}\left(-2\left|\nabla x_{\alpha}\right|_{g}^{2} u_{i} \mathfrak{L}_{\nu} u_{i}+4 u_{i} \mathfrak{L}_{\nu} x_{\alpha}\left\langle\nabla x_{\alpha}, \nabla u_{i}\right\rangle_{g}\right.\\&\quad\quad\quad\quad\quad\left.+4\left\langle\nabla x_{\alpha}, \nabla u_{i}\right\rangle_{g}^{2}+u_{i}^{2}\left( \mathfrak{L}_{\nu} x_{\alpha}\right)^{2}\right) e^{\langle\nu,X\rangle_{g_{0}}}dv.
\end{aligned}\end{equation}
From \eqref{v-2}, \eqref{xx-v} and \eqref{nH-ine}, we have

\begin{equation}\begin{aligned}\label{vxa}
\sum_{\alpha=1}^{n+p}\left( \mathfrak{L}_{\nu} x_{\alpha}\right)^{2}&=\sum_{\alpha=1}^{n+p}\left(\Delta x_{\alpha}+\left\langle\nabla x_{\alpha}, \nu\right\rangle_{g_{0}}\right)^{2} \\&
=\sum_{\alpha=1}^{n+p}\left(\left(\Delta x_{\alpha}\right)^{2}+2 \Delta x_{\alpha}\left\langle\nabla x_{\alpha}, \nu\right\rangle_{g_{0}}+\left\langle\nabla x_{\alpha}, \nu\right\rangle_{g_{0}}^{2}\right) \\&
= n^{2}H^{2}+|\nu^{\top}|_{g_{0}}^{2}.
\end{aligned}\end{equation}
From \eqref{n-ine},  \eqref{vxx-u},  \eqref{na-ui-2}, \eqref{vxa}  and \eqref{sum-Psi-1}, we get

\begin{equation*}
\begin{aligned}\sum^{n+p}_{\alpha=1}\widehat{\Psi}_{i,\alpha}&\leq\int_{\Omega}\left(-2 n u_{i} \mathfrak{L}_{\nu} u_{i}+4 u_{i}
\left|\nabla u_{i}\right|_{g}\left|\nu^{\top}\right|_{g_{0}}
+4\left|\nabla u_{i}\right|_{g}^{2}+u_{i}^{2}\left(n^{2}H^{2}+|\nu^{\top}|_{g_{0}}^{2}\right)\right) e^{\langle\nu,X\rangle_{g_{0}}}dv.
\end{aligned}\end{equation*}
By Cauchy-Schwarz inequality, we have
\begin{equation}\label{Cau-Sch}
\begin{aligned}\sum^{n+p}_{\alpha=1}\widehat{\Psi}_{i,\alpha}&\leq\int_{\Omega}\left[-2 n u_{i} \mathfrak{L}_{\nu} u_{i} +4\left|\nabla u_{i}\right|_{g}^{2}+u_{i}^{2}\left(n^{2}H^{2}+|\nu^{\top}|_{g_{0}}^{2}\right)\right] e^{\langle\nu,X\rangle_{g_{0}}}dv\\&\ \ \ \ +4\left(\int_{\Omega}(u_{i}|\nu^{\top}|_{g_{0}})^{2} e^{\langle\nu,X\rangle_{g_{0}}}dv\right)^{\frac{1}{2}}\left(\int_{\Omega}|\nabla u_{i}|_{g}^{2} e^{\langle\nu,X\rangle_{g_{0}}}dv\right)^{\frac{1}{2}}.
\end{aligned}\end{equation} By divergence theorem and Cauchy-Schwarz inequality, we conclude that

\begin{equation}\begin{aligned}\label{cs-ineq}
\int_{\Omega}\left|\nabla u_{i}\right|_{g}^{2} e^{\langle\nu,X\rangle_{g_{0}}}dv&
=-\int_{\Omega} u_{i} \mathfrak{L}_{\nu} u_{i} e^{\langle\nu,X\rangle_{g_{0}}}dv \\&\leq\left\{\int_{\Omega} u_{i}^{2} e^{\langle\nu,X\rangle_{g_{0}}}dv
\right\}^{\frac{1}{2}}\left\{\int_{\Omega}\left( \mathfrak{L}_{\nu} u_{i}\right)^{2} e^{\langle\nu,X\rangle_{g_{0}}}dv\right\}^{\frac{1}{2}}\\&=\Lambda_{i}^{\frac{1}{2}}.
\end{aligned}\end{equation}
Thus, from \eqref{Cau-Sch} and \eqref{cs-ineq}, we have

\begin{equation*}
\begin{aligned}\sum^{n+p}_{\alpha=1}\widehat{\Psi}_{i,\alpha}&\leq\int_{\Omega}\left[-2 n u_{i} \mathfrak{L}_{\nu} u_{i} +4\left|\nabla u_{i}\right|_{g}^{2}+u_{i}^{2}\left(n^{2}H^{2}+|\nu^{\top}|_{g_{0}}^{2}\right)\right]e^{\langle\nu,X\rangle_{g_{0}}}dv\\&\ \ \ \ +4 \Lambda_{i}^{\frac{1}{4}}\left(\int_{\Omega}(u_{i}|\nu^{\top}|_{g_{0}})^{2} e^{\langle\nu,X\rangle_{g_{0}}}dv\right)^{\frac{1}{2}},
\end{aligned}\end{equation*}
which gives \eqref{sum-W-Psi}. Thus, it finishes the proof of lemma \ref{lem-Psi}.

\end{proof}
\begin{lem}\label{lem-theta}Let $x_{1}, x_{2}, \ldots, x_{n+p}$ be the standard coordinate functions of $\mathbb{R}^{n+p}$. For any $i=1,2,\cdots k$ and $\alpha=1,2,\cdots,n+p$, let $$\widehat{\Theta}_{i,\alpha}:=\int_{\Omega}\Theta_{i}(x_{\alpha}) e^{\langle\nu,X\rangle_{g_{0}}}dv,$$
where function $\Psi_{i}$ is given by \eqref{Theta}. Then, we have

\begin{equation}\begin{aligned}\label{sum-theda-iden} \sum^{n+p}_{\alpha=1}\widehat{\Theta}_{i,\alpha}
&\leq\int_{\Omega}\left[\left|\nabla u_{i}\right|_{g}^{2}+\frac{1}{4} u_{i}^{2}\left(n^{2}H^{2}+|\nu^{\top}|_{g_{0}}^{2}\right)\right] e^{\langle\nu,X\rangle_{g_{0}}}dv\\&\ \ \ \ +\Lambda_{i}^{\frac{1}{4}}\left[\int_{\Omega}(u_{i}|\nu^{\top}|_{g_{0}})^{2} e^{\langle\nu,X\rangle_{g_{0}}}dv\right]^{\frac{1}{2}}.\end{aligned}\end{equation}

\end{lem}
\begin{proof}
Taking $f=x_{\alpha}$ in \eqref{Theta} and summing over $\alpha$ from 1 to $n+p$, we get

\begin{equation}\sum^{n+p}_{\alpha=1}\widehat{\Theta}_{i,\alpha}=\sum^{n+p}_{\alpha=1}\int_{\Omega}\left(\left\langle\nabla x_{\alpha}, \nabla u_{i}\right\rangle_{g}
+\frac{u_{i} \mathfrak{L}_{\nu} x_{\alpha}}{2}\right)^{2} e^{\langle\nu,X\rangle_{g_{0}}}dv.
\end{equation}
Utilizing  \eqref{vxx-u},  \eqref{na-ui-2}, \eqref{vxa}  and \eqref{sum-Psi-1}, we infer that

\begin{equation}\begin{aligned}\label{no-sum-Theta}
&\sum_{\alpha=1}^{n+p}\left(\left\langle\nabla x_{\alpha}, \nabla u_{i}\right\rangle_{g}
+\frac{u_{i} \mathfrak{L}_{\nu} x_{\alpha}}{2}\right)^{2}\\&=\sum_{\alpha=1}^{n+p}\left[\left\langle\nabla x_{\alpha}, \nabla u_{i}\right\rangle_{g}^{2}+u_{i} \mathfrak{L}_{\nu} x_{\alpha}\left\langle\nabla x_{\alpha}, \nabla u_{i}\right\rangle_{g}+\frac{1}{4}\left(u_{i} \mathfrak{L}_{\nu} x_{\alpha}\right)^{2}\right]\\&
\leq\left|\nabla u_{i}\right|_{g}^{2}+u_{i}\left|\nabla u_{i}\right|_{g}|\nu^{\top}|_{g_{0}}+\frac{1}{4} u_{i}^{2}\left(n^{2}H^{2}+|\nu^{\top}|_{g_{0}}^{2}\right).
\end{aligned}\end{equation}
By the definition of $\widehat{\Theta}_{i,\alpha}$ and \eqref{no-sum-Theta}, we obtain

\begin{equation}\begin{aligned}\sum^{n+p}_{\alpha=1}\widehat{\Theta}_{i,\alpha}\leq\int_{\Omega}\left[\left|\nabla u_{i}\right|_{g}^{2}+u_{i}\left|\nabla u_{i}\right|_{g}
|\nu^{\top}|_{g_{0}}+\frac{1}{4} u_{i}^{2}\left(n^{2}H^{2}+|\nu^{\top}|_{g_{0}}^{2}\right)\right] e^{\langle\nu,X\rangle_{g_{0}}}dv.\end{aligned}\end{equation}
By Cauchy-Schwarz inequality, we infer that

\begin{equation}\begin{aligned}\label{div-sum-Theta}\sum^{n+p}_{\alpha=1}\widehat{\Theta}_{i,\alpha}
&\leq\int_{\Omega}\left[\left|\nabla u_{i}\right|_{g}^{2}+\frac{1}{4} u_{i}^{2}\left(n^{2}H^{2}+|\nu^{\top}|_{g_{0}}^{2}\right)\right]e^{\langle\nu,X\rangle_{g_{0}}}dv\\&\ \ \ \ +\left(\int_{\Omega}(u_{i}|\nu^{\top}|_{g_{0}})^{2} e^{\langle\nu,X\rangle_{g_{0}}}dv\right)^{\frac{1}{2}}\left(\int_{\Omega}|\nabla u_{i}|_{g}^{2} e^{\langle\nu,X\rangle_{g_{0}}}dv\right)^{\frac{1}{2}}.\end{aligned}\end{equation}
 From \eqref{cs-ineq} and \eqref{div-sum-Theta}, we obtain

\begin{equation*}\begin{aligned} \sum^{n+p}_{\alpha=1}\widehat{\Theta}_{i,\alpha}
&\leq\int_{\Omega}\left[\left|\nabla u_{i}\right|_{g}^{2}+\frac{1}{4} u_{i}^{2}\left(n^{2}H^{2}+|\nu^{\top}|_{g_{0}}^{2}\right)\right]e^{\langle\nu,X\rangle_{g_{0}}}dv\\&\ \ \ \ +\Lambda_{i}^{\frac{1}{4}}\left(\int_{\Omega}(u_{i}|\nu^{\top}|_{g_{0}})^{2} e^{\langle\nu,X\rangle_{g_{0}}}dv\right)^{\frac{1}{2}}.\end{aligned}\end{equation*}
Therefore, we finish the proof  of lemma \ref{lem-theta}.

\end{proof}

In order to prove theorem \ref{thm1.1}, we need the following embedding theorem due to Nash.

\begin{Nash}  Each complete Riemannian manifold $\mathcal{M}^{n}$ can be isometrically immersed
into a Euclidian space  $\mathbb{R}^{n+p}$.\end{Nash}

Next, we would like to give the proof of theorem \ref{thm1.1}.

\noindent\emph{Proof of theorem} \ref{thm1.1}. Since $\mathcal{M}^{n}$ is a complete Riemannian manifold, Nash's theorem implies that there exists an isometric immersion from $M^{n}$ into a Euclidean space $\mathbb{R}^{n+p}.$ Thus, $\mathcal{M}^{n}$ can be considered as an $n$ -dimensional complete isometrically immersed submanifold in $\mathbb{R} ^{n+p}$.
According to lemma \ref{lemma2.1}, the definitions of $\widehat{\Psi}_{i,\alpha}$  and $\widehat{\Theta}_{i,\alpha} $, we have

\begin{equation}\begin{aligned}\label{Sum-L1}
\sum_{i=1}^{k}\left(\Lambda_{k+1}-\Lambda_{i}\right)^{2} \int_{\Omega} u_{i}^{2}|\nabla x_{\alpha} |^{2} e^{\langle\nu,X\rangle_{g_{0}}}dv
&\leq \sum_{i=1}^{k}\delta\left(\Lambda_{k+1}-\Lambda_{i}\right)^{2}
\widehat{\Psi}_{i,\alpha} \\&+\sum_{i=1}^{k}\frac{\left(\Lambda_{k+1}-\Lambda_{i}\right)}{\delta} \widehat{\Theta}_{i,\alpha}.
\end{aligned}\end{equation}By \eqref{n-ine}, we have
 \begin{equation}\label{left-1} \sum_{\alpha=1}^{n+p}\int_{\Omega} u_{i}^{2}\left|\nabla x_{\alpha}\right|_{g}^{2} e^{\langle\nu,X\rangle_{g_{0}}}dv=n. \end{equation}
Using \eqref{left-1}, and summing over $\alpha$ from 1 to $n+p$ for \eqref{Sum-L1}, one has

\begin{equation}\begin{aligned}\label{SUM-1-1}
n\sum_{i=1}^{k}\left(\Lambda_{k+1}-\Lambda_{i}\right)^{2}&
\leq \sum_{i=1}^{k}\sum_{\alpha=1}^{n+p}\delta\left(\Lambda_{k+1}-\Lambda_{i}\right)^{2}
\widehat{\Psi}_{i,\alpha} +\sum_{i=1}^{k} \sum_{\alpha=1}^{n+p} \frac{\left(\Lambda_{k+1}-\Lambda_{i}\right)}{\delta} \widehat{\Theta}_{i,\alpha}\\&=\sum_{i=1}^{k}\delta\left(\Lambda_{k+1}-\Lambda_{i}\right)^{2}
\sum_{\alpha=1}^{n+p}\widehat{\Psi}_{i,\alpha} +\sum_{i=1}^{k} \frac{\left(\Lambda_{k+1}-\Lambda_{i}\right)}{\delta} \sum_{\alpha=1}^{n+p} \widehat{\Theta}_{i,\alpha}.
\end{aligned}\end{equation}Next, we estimate the upper bounds for $\widehat{\Psi}_{i,\alpha}$  and $\widehat{\Theta}_{i,\alpha} $.
Since eigenvalues are invariant under isometries, letting

\begin{equation*}
C_{1}=\frac{1}{4}\inf _{\sigma \in \Pi}\max_{\Omega}\left(n^{2}H^{2}\right),
\end{equation*}and

$$\widetilde{C}_{1}=\frac{1}{4}\max_{\Omega} |\nu^{\top}|_{g_{0}} ,$$
where $\Pi$ denotes the set of all isometric immersions from $\mathcal{M}^{n}$ into a Euclidean space, it follows from \eqref{sum-W-Psi}, \eqref{cs-ineq} and \eqref{sum-theda-iden} that,

\begin{equation} \label{wid-PSI} \sum_{\alpha=1}^{n+p}\widehat{\Psi}_{i,\alpha}\leq(2 n+4) \Lambda_{i}^{\frac{1}{2}}+4\left(4\widetilde{C}_{1}\Lambda_{1}^{\frac{1}{4}}+4\widetilde{C}_{1}^{2}+C_{1}\right),\end{equation}and

\begin{equation} \label{wid-THETA} \sum_{\alpha=1}^{n+p}\widehat{\Theta}_{i,\alpha}\leq  \Lambda_{i}^{\frac{1}{2}}+4\widetilde{C}_{1}\Lambda_{1}^{\frac{1}{4}}+4\widetilde{C}_{1}^{2}+C_{1}.\end{equation}
Substituting \eqref{wid-PSI} and \eqref{wid-THETA} into \eqref{SUM-1-1}, we have

\begin{equation*}\begin{aligned}
n \sum_{i=1}^{k}\left(\Lambda_{k+1}-\Lambda_{i}\right)^{2}&
\leq  \sum_{i=1}^{k} \delta\left(\Lambda_{k+1}-\Lambda_{i}\right)^{2}\left[(2 n+4) \Lambda_{i}^{\frac{1}{2}}+4\overline{C}_{1}\right]\\&\quad+\sum_{i=1}^{k} \frac{\Lambda_{k+1}-\Lambda_{i}}{\delta}\left(\Lambda_{i}^{\frac{1}{2}}+\overline{C}_{1}\right),
\end{aligned}\end{equation*}where

$$\overline{C}_{1}=4\widetilde{C}_{1}\Lambda_{1}^{\frac{1}{4}}+4\widetilde{C}_{1}^{2}+C_{1}.$$
In above inequality, taking

\begin{equation*}
\delta=\frac{\left[\sum_{i=1}^{k}\left(\Lambda_{k+1}-\Lambda_{i}\right)\left(\Lambda_{i}^{\frac{1}{2}}+\overline{C}_{1}\right)\right]^{\frac{1}{2}}}{\left[\sum_{i=1}^{k}\left(\Lambda_{k+1}-\Lambda_{i}\right)^{2}\left((2 n+4) \Lambda_{i}^{\frac{1}{2}}+4\overline{C}_{1}\right)\right]^{\frac{1}{2}}},
\end{equation*}
we can get \eqref{thm-1.1-ineq}. Hence, we finish the proof of theorem \ref{thm1.1}.
$$\eqno\Box$$Next, we give the proof of theorem \ref{thm1.2}.
\vskip 3mm

\noindent\emph{Proof of theorem} \ref{thm1.2}. Since $\mathcal{M}^{n}$ is an $n$-dimensional  complete  translator isometrically embedded into the Euclidean space $\mathbb{R}^{n+p}$, we have \begin{equation}\label{3.6-1}\textbf{H}=\nu_{0}^{\perp}, \end{equation} and

\begin{equation}\label{3.6-2} |\nu_{0}|^{2}=1,\end{equation}  which implies that \begin{equation}\label{3.6-7}n^{2}H^{2}+|\nu_{0}^{\top}|^{2}= n^{2}|\nu_{0}^{\perp}|^{2}+|\nu_{0}^{\top}|^{2}\leq n^{2},\end{equation}and

\begin{equation}\label{3.6-3}|\nu_{0}^{\top}|\leq1.
\end{equation}
Uniting \eqref{3.6-1}, \eqref{3.6-2}, \eqref{3.6-3} and \eqref{3.6-7}, we yield

\begin{equation}\label{3.8}\frac{1}{4}\int_{\Omega}u_{i}^{2}\left(n^{2}H^{2} +|\nu_{0}^{\top}|^{2}+4\Lambda^{\frac{1}{4}}_{i}|\nu_{0}^{\top}|\right)e^{\langle\nu_{0},X\rangle_{g_{0}}}dv\leq \frac{1}{4}\max_{\Omega}\left(n^{2}+4\Lambda^{\frac{1}{4}}_{i}\right).\end{equation} Substituting \eqref{3.8} into \eqref{thm-1.1-ineq}, we obtain \eqref{thm-1.1-ineq-trans}. Therefore, we finish the proof of theorem \ref{thm1.2}.
$$\eqno\Box$$

\section{Some Further Applications}\label{sec4}
In this section, we would like to give some applications of our theorem \ref{thm1.1}. Specially, we establish the eigenvalue inequalities on the minimal submanifolds of the
Euclidean space, submanifolds on the unit sphere and projective spaces.

First, let us suppose that $(\mathcal{M}^{n},g)$ is an $n$-dimensional complete minimal submanifold isometrically embedded into the Euclidean space $\mathbb{R}^{n+p}$, then the mean curvature vanishes. Furthermore, by theorem \ref{thm1.1}, we can deduce the following corollary.
\begin{corr}\label{corr1.1}
Let $(\mathcal{M}^{n},g)$ be an $n$-dimensional complete minimal submanifold isometrically embedded into the Euclidean space $\mathbb{R}^{n+p}$, then
eigenvalues $\Lambda
_{i}$ of the clamped plate problem \eqref{L-2-prob} of  the $\mathfrak{L}_{\nu}^{2}$ operator satisfy
\begin{equation}\label{thm-1.1-ineq-1}\begin{aligned}
\sum_{i=1}^{k}\left(\Lambda_{k+1}-\Lambda_{i}\right)^{2}
\leq & \frac{4}{n}\left\{\sum_{i=1}^{k}\left(\Lambda_{k+1}-\Lambda_{i}\right)^{2}\left(\left(\frac{n}{2}+1\right) \Lambda_{i}^{\frac{1}{2}}+4\widetilde{C}_{1}\Lambda_{1}^{\frac{1}{4}}+4\widetilde{C}_{1}^{2}\right)\right\}^{\frac{1}{2}} \\
& \times\left\{\sum_{i=1}^{k}\left(\Lambda_{k+1}-\Lambda_{i}\right)\left(\Lambda_{i}^{\frac{1}{2}}+4\widetilde{C}_{1}\Lambda_{1}^{\frac{1}{4}}+4\widetilde{C}_{1}^{2}\right)\right\}^{\frac{1}{2}},
\end{aligned}\end{equation}
where $C_{1}$ is a constant given by
$$
C_{1}=\frac{1}{4}\max_{\Omega}|\nu^{\top}|_{g_{0}}.
$$
\end{corr}

Suppose  that $(\mathcal{M}^{n},g)$ is an $n$-dimensional submanifold isometrically immersed in the unit sphere $\mathbb{S}^{n+p-1}(1) \subset \mathbb{R}^{n+p}$  with mean curvature vector $ \overline{\textbf{H}}$. Let $\Omega$ be a bounded domain on the Riemannian manifolds $\mathcal{M}^{n}$. Next, we consider the eigenvalue problem of $\mathcal{L}_{\nu}^{2}$ operator on the Riemannian manifolds $\mathcal{M}^{n}$.  By theorem \ref{thm1.1}, we have the following corollary.
\begin{corr}\label{corr-6.2}
If $(\mathcal{M}^{n},g)$ be an $n$-dimensional submanifold isometrically immersed in the unit sphere $\mathbb{S}^{n+p-1}(1) \subset \mathbb{R}^{n+p}$  with mean curvature vector $\overline{\textbf{H}}$. Then,
eigenvalues of the clamped plate problem \eqref{L-2-prob} of  the $\mathfrak{L}_{\nu}^{2}$ operator satisfy
\begin{equation}\begin{aligned}\label{Sub-Sph}
\sum_{i=1}^{k}\left(\Lambda_{k+1}-\Lambda_{i}\right)^{2} \leq &\frac{4}{n} \left\{\sum_{i=1}^{k}\left(\Lambda_{k+1}-\Lambda_{i}\right)^{2}\left[\left( \frac{n}{2}+1\right) \Lambda_{i}^{\frac{1}{2}}+4\widetilde{C}_{2} \Lambda_{i}^{\frac{1}{4}}+4\widetilde{C}_{2}^{2}+C_{2}\right]\right\}^{\frac{1}{2}}\\
&\quad\quad\times\left\{ \sum_{i=1}^{k}\left(\Lambda_{k+1}-\Lambda_{i}\right)\left(\Lambda_{i}^{\frac{1}{2}}+ 4\widetilde{C} _{2} \Lambda_{i}^{\frac{1}{4}}+4\widetilde{C}_{2}^{2}+C_{2}\right)\right\}^{\frac{1}{2}},
\end{aligned}\end{equation}where

$$C_{2}=\frac{1}{4}\inf_{\sigma\in\Pi}\max_{\Omega}n^{2}(|\overline{\textbf{H}}|^{2}+1),$$
and

$$\widetilde{C}_{2}=\frac{1}{4}\max_{\Omega}|\nu^{\top}|_{g_{0}}.$$
\end{corr}

\begin{proof}
Clearly, there exists a canonical  imbedding map from the unit sphere to Euclidean space. We assume that $\pi: \mathbb{S}^{n+p-1}(1)\rightarrow \mathbb{R}^{n+p}$ is the canonical imbedding from the unit sphere $\mathbb{S}^{n+p-1}(1)$ into $\mathbb{R}^{n+p},$ and  $\tau: \mathcal{M}^{n}\rightarrow \mathbb{S} ^{n+p-1}(1)$ is an isometrical immersion.  In other words, we have
the following diagram:

\begin{equation*}\begin{aligned}
\xymatrix{
  \mathcal{M}^{n}\ar[dr]_{\pi\circ \tau} \ar[r]^{\tau}
                & \mathbb{S}^{n+p-1} \ar[d]^{\pi}  \\
                & \mathbb{R}^{n+p}              }\end{aligned} \end{equation*}Then, $\pi \circ \tau: \mathcal{M}^{n} \rightarrow \mathbb{R}^{n+p}$ is an isometric immersion from $\mathcal{M}^{n}$ to $\mathbb{R}^{n+p} .$
Let $\overline{\textbf{H}}$ and $\textbf{H}$ be the mean curvature vector fields of $\tau$ and $\pi \circ \tau,$ respectively. Then, we have
\[
\left| \textbf{H}\right|^{2}=|\overline{\textbf{H}}|^{2}+1.
\]
Applying  theorem \ref{thm1.1} directly, we can get \eqref{Sub-Sph}. Therefore, we finish the proof of corollary \ref{corr-6.2}.\end{proof}

In particular, we assume that $(\mathcal{M}^{n},g)$ is an $n$-dimensional unit sphere $\mathbb{S}^{n}(1)$, and then, $|\overline{\textbf{H}}|=0$, which implies that, $\left| \textbf{H}\right|=1$. Thus, according to corollary \ref{corr-6.2}, we can prove the following corollary.

\begin{corr}Let $(\mathcal{M}^{n},g)$ be an $n$-dimensional unit sphere $\mathbb{S}^{n}(1)$ and $\Omega$ is a bounded domain on  $\mathbb{S}^{n}(1)$. Then,
eigenvalues of the clamped plate problem \eqref{L-2-prob} of  the $\mathfrak{L}_{\nu}^{2}$ operator satisfy

\begin{equation*}\begin{aligned}
\sum_{i=1}^{k}\left(\Lambda_{k+1}-\Lambda_{i}\right)^{2}&\leq\frac{4}{n}\left\{\sum_{i=1}^{k}\left(\Lambda_{k+1}-\Lambda_{i}\right)^{2}\left[\left(\frac{n}{2}+1\right) \Lambda_{i}^{\frac{1}{2}}+4\Gamma^{\frac{1}{4}}_{1}C_{3}+\frac{n^{2}}{4}+4C_{3}^{2}\right]\right\}^{\frac{1}{2}}\\
&\quad\quad\times\left\{\sum_{i=1}^{k}\left(\Lambda_{k+1}-\Lambda_{i}\right)\left( \Lambda_{i}^{\frac{1}{2}}+4\Gamma^{\frac{1}{4}}_{1}C_{3}+\frac{n^{2}}{4}+4C_{3}^{2}\right)\right\}^{\frac{1}{2}},
\end{aligned}\end{equation*}
where $C_{3}$ is given by

$$
C_{3}=\frac{1}{4}\max_{\Omega}|\nu^{\top}|_{g_{0}}.
$$
\end{corr}
$$\eqno\Box$$

In what follows, let us recall some important facts for the  submanifolds on the projective spaces. For more useful information, we refer to \cite{Chb,CL}. We assume that $\mathbb{F}$ is the field $\mathbb{R}$ of real numbers,
the field $\mathbb{C}$ of complex numbers or the field $\mathbb{Q}$ of quaternions. For convenience, we
introduce a notation as follows:

\begin{equation}\label{df}
d_{\mathbb{F}}=\operatorname{dim}_{\mathbb{R}}\mathbb{F}=\left\{\begin{array}{ll}
1, & \text { if } \mathbb{F} = \mathbb{R}; \\
2, & \text { if } \mathbb{F} = \mathbb{C}; \\
4, & \text { if } \mathbb{F} = \mathbb{Q}.
\end{array}\right.
\end{equation}Denote by $\mathbb{F}P^{m}$ the $m$-dimensional real projective space if $\mathbb{F}= \mathbb{R} ,$ the complex projective space of real dimension $2 m$ if $\mathbb{F}= \mathbb{C}$, and the quaternionic projective space of real dimension $4 m$ if $\mathbb{F}= \mathbb{Q}$, respectively. Here, the manifold $\mathbb{F}P^{m}$ carries a natural metric so that the Hopf fibration $\pi: \mathbb{S}^{d_{\mathbb{F}} \cdot(m+1)-1} \subset \mathbb{F}^{m+1} \rightarrow \mathbb{F}P^{m}$ is a Riemannian fibration.
Let $\mathcal{H}_{m+1}(\mathbb{F})=\left\{A \in \mathcal{M}^{n} _{m+1}(\mathbb{F}) \mid A^{*}:=\overline{^{t} A}=A\right\}$ be the vector space of $(m+1) \times(m+1)$ Hermitian
matrices with coefficients in the field $\mathbb{F}$. We can endow $\mathcal{H}_{m+1}( \mathbb{F})$ with the inner product
\[
\langle A, B\rangle=\frac{1}{2} \operatorname{tr}(A B),
\]
where tr $(\cdot)$ denotes the trace for the given $(m+1) \times(m+1)$ matrix. Clearly, the map $\pi: \mathbb{S} ^{d_{\mathbb{F}} \cdot(m+1)-1} \subset \mathbb{F} ^{m+1} \rightarrow$
$\mathcal{H}_{m+1}(\mathbb{F})$ given by
\[
\pi(\textbf{z})=\textbf{z}\textbf{z}^{\ast}=\left(\begin{array}{llll}
\left|z_{0}\right|^{2} & z_{0} \overline{z_{1}} & \cdots & z_{0} \overline{z_{m}} \\
z_{1} \overline{z_{0}} & \left|z_{1}\right|^{2} & \cdots & z_{1} \overline{z_{m}} \\
\cdots & \cdots & \cdots & \cdots \\
z_{m} \overline{z_{0}} & z_{m} \overline{z_{1}} & \cdots & \left|z_{m}\right|^{2}
\end{array}\right)
\]
induces through the Hopf fibration an isometric embedding $\pi$ from $\mathbb{F}P^{m}$ into $\mathcal{H}_{m+1}( \mathbb{F}),$ where $\textbf{z}=(z_{0},z_{1},\cdots,z_{m})\in\mathbb{S}^{d_{\mathbb{F}} \cdot(m+1)-1}$. Moreover, $\pi\left( \mathbb{F}P^{m}\right)$ is a minimal submanifold of the hypersphere $\mathbb{S} \left(\frac{I}{m+1}, \sqrt{\frac{m}{2(m+1)}}\right)$ of $\mathcal{H}_{m+1}( \mathbb{F})$ with radius $\sqrt{\frac{m}{2(m+1)}}$ and
center $\frac{I}{m+1}$, where $I$ denotes the identity matrix.
In addition, we need the follow result (cf. lemma 6.3 in Chapter 4 in \cite{Chb}):

\begin{lem} \label{lem-proj}Let $\rho: \mathcal{M}^{n}  \rightarrow \mathbb{F} P^{\text {m }}$ be an isometric immersion, and let $\widehat{\textbf{H}}$ and $\textbf{H}$ be the mean curvature
vector fields of the immersions $\rho$ and $\pi \circ \rho,$ respectively (here $\pi$ is the induced isometric embedding $\pi$ from $\mathbb{F}P^{m}$ into $\mathcal{H}_{m+1}( \mathbb{F})$ explained above). Then, we have
\[
\left| \textbf{H}\right|^{2}=|\widehat{\textbf{H}}|^{2}+\frac{4(n+2)}{3 n}+\frac{2}{3 n^{2}} \sum_{i \neq j} K\left(e_{i}, e_{j}\right),
\]
where $\left\{e_{i}\right\}_{i=1}^{n}$ is a local orthonormal basis of $\Gamma(T \mathcal{M}^{n})$ and $K$ is the sectional curvature of $\mathbb{F}P^{m}$ expressed $b y$
\[
K\left(e_{i}, e_{j}\right)=\left\{\begin{array}{ll}
1, & \text { if } \mathbb{F} = \mathbb{R}; \\
1+3\left(e_{i} \cdot J e_{j}\right)^{2}, & \text { if } \mathbb{F} = \mathbb{C}; \\
1+\sum_{r=1}^{3} 3\left(e_{i} \cdot J_{r} e_{j}\right)^{2}, & \text { if } \mathbb{F} = \mathbb{Q},
\end{array}\right.
\]
where $J$ is the complex structure of $\mathbb{C}P^{m}$ and $J_{r}$ is the quaternionic structure of $\mathbb{Q}P ^{m}$.\end{lem}

$$\eqno\Box$$

Therefore, one can infer from lemma \ref{lem-proj} that

\begin{equation}\label{H-3H}
\left|\textbf{H}\right|^{2}=\left\{\begin{array}{ll}
|\widehat{\textbf{H}}|^{2}+\frac{2(n+1)}{2 n}, & \text { for } \mathbb{R} P^{m}; \\
|\widehat{\textbf{H}}|^{2}+\frac{2(n+1)}{2 n}+\frac{2}{n^{2}} \sum_{i, j=1}^{n}\left(e_{i} \cdot J e_{j}\right)^{2} \leq|\widehat{\textbf{H}}|^{2}+\frac{2(n+2)}{n}, & \text { for } \mathbb{C} P^{m}; \\
|\widehat{\textbf{H}}|^{2}+\frac{2(n+1)}{2 n}+\frac{2}{n^{2}} \sum_{i, j=1}^{n} \sum_{r=1}^{3}\left(e_{i} \cdot J_{r} e_{j}\right)^{2} \leq|\widehat{\textbf{H}}|^{2}+\frac{2(n+4)}{n}, & \text { for } \mathbb{Q} P^{m}.
\end{array}\right.
\end{equation}Hence, it follows from \eqref{H-3H} that,

\begin{equation}\label{HH}
\left| \textbf{H}\right|^{2} \leq|\widehat{\textbf{H}}|^{2}+\frac{2\left(n+d_{\mathbb{F}}\right)}{n}.
\end{equation}
We note that the equality in \eqref{HH} holds if and only if $\mathcal{M}^{n}$ is a complex submanifold of $\mathbb{C}P^{m}$ (for the case $\mathbb{C}P^{m}$ ) while $n \equiv 0(\bmod 4)$ and $\mathcal{M}^{n}$ is an invariant submanifold of $\mathbb{Q}P^{m}\left(\text { for the case } \mathbb{Q}P^{m}\right)$.   We use $\widehat{\Pi}$ to denote the set of all isometric immersions from $\mathcal{M}^{n}$ into a projective space $\mathbb{F}P^{m}$. Then, applying theorem \ref{thm1.1}, we can prove the following corollary.

\begin{corr}\label{corr-6.4}
If $\mathcal{M}^{n}$ is isometrically immersed in a projective space $\mathbb{F}P^{m}$ with mean curvature vector $\widehat{\textbf{H}}$, Then,
eigenvalues of the clamped plate problem \eqref{L-2-prob} of  the $\mathfrak{L}_{\nu}^{2}$ operator satisfy

\begin{equation}\label{proj-inequa}
\begin{aligned}
\sum_{i=1}^{k}\left(\Lambda_{k+1}-\Lambda_{i}\right)^{2} &\leq \frac{4}{n}\left\{\sum_{i=1}^{k}\left(\Lambda_{k+1}-\Lambda_{i}\right)^{2}\left[\left( \frac{n}{2}+1\right) \Lambda_{i}^{\frac{1}{2}}+4\widetilde{C}_{4} \Lambda_{i}^{\frac{1}{4}}+4\widetilde{C}_{4}^{2}+C_{4}\right]\right\}^{\frac{1}{2}}\\&\quad\quad\times
\left\{\sum_{i=1}^{k}\left(\Lambda_{k+1}-\Lambda_{i}\right)  \left(\Lambda_{i}^{\frac{1}{2}}+ 4\widetilde{C} _{4} \Lambda_{i}^{\frac{1}{4}}+4\widetilde{C}_{4}^{2}+C_{4}\right)
\right\}^{\frac{1}{2}},
\end{aligned}
\end{equation}
where
where $C_{4}$ is given by

$$
C_{4}=\frac{1}{4}\inf _{\widehat{\sigma} \in \widehat{\Pi}}\max_{\Omega}\left(n^{2}H^{2}+2n\left(n+d_{\mathbb{F} }\right)\right),
$$  and $\widetilde{C}_{4}$ is given by $$\widetilde{C}_{4}=\frac{1}{4}\max_{\Omega} |\nu^{\top}|_{g_{0}},$$ and
$d_{\mathbb{F}}=\operatorname{dim}_{ \mathbb{R} } \mathbb{F}$ defined by \eqref{df}.

\end{corr}

\begin{proof}
Since there is a canonical imbedding from $\mathbb{F} P^{m}( \mathbb{F} = \mathbb{R} , \mathbb{C} , \mathbb{Q} )$ to Euclidean space $\mathcal{H}_{m+1}( \mathbb{F} )$, then for compact manifold $\mathcal{M}^{n}$ isometrically immersed into the projective space $\mathbb{F} P^{m},$ we have the following diagram:

\begin{equation*}\begin{aligned}
\xymatrix{
  \mathcal{M}^{n}\ar[dr]_{\pi\circ \rho} \ar[r]^{\rho}
                & \mathbb{F}P^{m} \ar[d]^{\pi}  \\
                &\mathcal{H}_{m+1}(\mathbb{F})             }\end{aligned} \end{equation*}
where $\pi: \mathbb{F}P^{m} \rightarrow \mathcal{H}_{m+1}( \mathbb{F})$ denotes the canonical imbedding from $\mathbb{F}P^{m}$ into $\mathcal{H}_{m+1}( \mathbb{F}),$ and   $\rho: \mathcal{M}^{n} \rightarrow$
$\mathbb{F}P^{m}$ denotes  an isometric immersion from $\mathcal{M}^{n}$ to $\mathbb{F}P^{m}$. Then, $\pi \circ \rho: \mathcal{M}^{n} \rightarrow \mathcal{H}_{m+1}( \mathbb{F})$ is an isometric immersion from $\mathcal{M}^{n}$ to $\mathcal{H}_{m+1}( \mathbb{F})$. Applying \eqref{HH} and theorem \ref{thm1.1}, one can get \eqref{proj-inequa}. Thus, it completes the proof of corollary \ref{corr-6.4}.

\end{proof}


\begin{thebibliography}{99}

\bibitem{AB1}
M. S. Ashbaugh and R. D. Benguria, More bounds on eigenvalue ratios
for Dirichlet Laplacians in $n$ dimension. SIAM J. Math. Anal., 1993, {\bf  24}
(6): 1622-1651. \DOI{10.1137/0524091}

\bibitem{AB2}
M. S. Ashbaugh and R. D. Benguria, A sharp bound for the ratio of
the first two eigenvalues of Dirichlet Laplacians and extensions.
Ann. of Math., 1992,{\bf 135}(3): 601-628. \DOI{10.2307/2946578}

\bibitem{AB3}
M. S. Ashbaugh and R. D. Benguria, A second proof of the
Payne-P\'{o}lya-Weinberger conjecture. Comm. Math. Phys., 1992, {\bf  147}
(1): 181-190.  \DOI{10.1007/BF02099533}

\bibitem{A1}
M. S. Ashbaugh, Isoperimetric and universal inequalities for eigenvalues. in Spectral theory and
geometry (Edinburgh,1998), E. B. Davies and Yu Safalov eds., London Math.Soc. Lecture Notes,
 {\bf 273} (1999), Cambridge Univ. Press, Cambridge: 95-139.


\bibitem{A2}
M. S. Ashbaugh, Universal eigenvalue bounds of Payne-Polya-Weinberger. Hile-Prottter, and H.C.
Yang, Proc. Indian Acad. Sci. Math. Sci., 2002, {\bf 112} (1): 3-30.  \DOI{10.1007/bf02829638}

\bibitem{AV}
S. B. Angenent and J. J. L. Velazquez, Asymptotic shape of cusp singularities in curve shortening. Duke
Math. J., 1995, {\bf 77}(1): 71-110. \DOI{10.1215/S0012-7094-95-07704-7}


\bibitem{Bran}
J. J. A. M. Brands, Bounds for the ratios of the first three membrane eigenvalues. Arch. Rational Mech.
Anal., 1964, {\bf  16 }(4): 265-268. \DOI{10.1007/BF00276187}

\bibitem{Chb}
B.Y. Chen,  Total Mean Curvature and Submanifolds of Finite Type. World Scientific, Singapore (1984)

\bibitem{CC}
D. Chen and Q.-M. Cheng, Extrinsic estimates for eigenvalues of the
Laplace operator. J. Math. Soc. Japan, 2008, {\bf 60}(2): 325-339. \DOI{10.1007/s00209-008-0376-8}


\bibitem{CL}
D. Chen  and H. Li, The sharp estimates for the first eigenvalue of Paneitz operator in 4-manifold. arXiv preprint
arXiv:1010.3102 (2010)

\bibitem{CZ}
D. Chen and T. Zheng, Bounds for ratios of the membrane eigenvalues. J. Diff. Eqns., 2011, {\bf 250}(3): 1575-1590. \DOI{10.1016/j.jde.2010.10.009}

\bibitem{ChQ}
Q. Chen and H. Qiu, Rigidity of self-shrinkers and translating solitons of mean curvature flows. Adv. Math., 2016, {\bf 294}: 517-531.
\DOI{10.1016/j.aim.2016.03.004}

\bibitem{CP}
Q.-M. Cheng and Y. Peng, Estimates for eigenvalues of $\mathfrak L$
operator on self-Shrinkers. Commu. Contemporary Math., 2013, {\bf 15}(06): 1350011. \DOI{10.1142/S0219199713500119}

\bibitem {CQian}
Z.-C. Chen and C.-L. Qian, Estimates for discrete spectrum of Laplacian operator with any
order, J. China Univ. Sci. Tech. 20 (1990), 259-266.

\bibitem {CQ}
Q.-M. Cheng and X. Qi, Eigenvalues of the Laplacian on Riemannian manifolds. Inter. J. Math., 2012, {\bf  23}(07): 1250067.
\DOI{10.1142/S0129167X1250067X}

\bibitem {CY1}
Q.-M. Cheng and H.-C. Yang, Estimates on eigenvalues of Laplacian, Math. Ann. 2005, {\bf 331}(2): 445-460. \DOI{10.1007/s00208-004-0589-z}

\bibitem {CY2}
Q.-M. Cheng and H.-C. Yang, Inequalities for eigenvalues of
Laplacian on domains and compact complex hypersurfaces in complex
projective spaces. J. Math. Soc. Japan, 2006, {\bf  58}(2): 545-561. \DOI{10.2969/jmsj/1149166788}

\bibitem {CY3}
Q.-M. Cheng and H.-C. Yang, Inequalities for eigenvalues of a clamped plate problem, Trans. Amer. Math. Soc., 2006, {\bf 358}: 2625-2635.
\DOI{10.1090/S0002-9947-05-04023-7}

\bibitem {CY4}
Q.-M. Cheng and H.-C. Yang, Bounds on eigenvalues of Dirichlet
Laplacian. Math. Ann., 2007, {\bf 337}(1): 159-175. \DOI{10.1007/s00208-006-0030-x}

\bibitem{Chit}
J. Clutterbuck, O. Schn\"{u}rer and F. Schulze, Stability of translating solutions to mean curvature flow. Calc.
Var. Par. Diff. Equs.,  2007, {\bf29}(3): 281-293. \DOI{10.1007/s00526-006-0033-1}

\bibitem{CM}
T. H. Colding and W. P. Minicozzi II, Generic mean curvature flow I;
Generic Singularities. Ann. of Math., 2012, {\bf 175} (2): 755-833. \DOI{10.4007/annals.2012.175.2.7}

\bibitem{HP}
F. N. Hile and  M. H. Protter, Inequalities for eigenvalues of the
Laplacian. Indiana Univ. Math. J., 1980, {\bf  29}(4): 523-538.  \DOI{10.1512/iumj.1980.29.29040}

\bibitem{HY}
G. N. Hile and R. Z. Yeh, Inequalities for eigenvalues of the biharmonic operator, Pacific J.
Math., 1984, {\bf 112}: 115-133. \DOI{10.2140/pjm.1984.112.115}

\bibitem{Hook}
S. M. Hook, Domain independent upper bounds for eigenvalues of elliptic operator, Trans.
Amer. Math. Soc., 1990, {\bf 318}: 615-642. \DOI{10.1090/S0002-9947-1990-0994167-2}

\bibitem{H}
G. Huisken, Asymptotic behavior for singularities of the mean
curvature flow. J. Diff. Geom., 1990, {\bf  31}(1): 285-299.  \DOI{10.4310/jdg/1214444099}

\bibitem{LP}
M. Levitin and L. Parnovski, Commutators, spectral trace identities,
and universal estimates for eigenvalues, J. Funct. Anal., 2002, \textbf{192}: 425-445.  \DOI{10.1006/jfan.2001.3913}

\bibitem{Mar}
P. Marcellini, Bounds for the third membrane eigenvalue. J. Diff. Eqns., 1980, {\bf 37}(3): 438-443. \DOI{10.1016/0022-0396(80)90108-4}

\bibitem{PPW1}
L. E. Payne, G. P\'{o}lya, and H. F. Weinberger, Sur le quotient de deux fr\'{e}quences propres
cons\'{e}cutives, Comptes Rendus Acad. Sci. Paris, 1955, {\bf 241}: 917-919.

\bibitem{PPW2}
L. E.Payne, G. P\'{o}lya  and H. F. Weinberger, On the ratio of
consecutive eigenvalues. J. Math. and Phys., 1956, {\bf  35}(1-4): 289-298. \DOI{10.1002/sapm1956351289}


\bibitem{R}
R. C. Reilly, On the first eigenvalue of the Laplacian for compact submanifolds of Euclidean space.
Comm. math. Helv., 1977, {\bf  52}(1): 525-533. \DOI{10.1007/BF02567385}

\bibitem{Sun}
H. J. Sun, Yang-type inequalities for weighted eigenvalues of a
second order uniformly elliptic operator with a nonnegative
potential. Proc. Amer. Math. Soc., 2010, {\bf  138} (8): 2827-2838. \DOI{10.1090/S0002-9939-10-10321-9}


\bibitem{SCY}
H. Sun, Q.-M. Cheng and H.-C. Yang, Lower order eigenvalues of Dirichlet Laplacian. Manuscripta Math., 2008
{\bf  125}(2):  139-156. \DOI{10.1007/s00229-007-0136-9}



\bibitem{WX1}
Q. Wang and C. Xia,  Universal bounds for eigenvalues of the biharmonic operator on Riemannian manifolds.
J. Funct. Anal., 2007, {\bf 245}: 334-352. \DOI{10.1016/j.jfa.2006.11.007}

\bibitem{WX2}
Q. Wang and C. Xia, Universal bounds for eigenvalues of the biharmonic operator. J. Math. Anal. Appl., 2010,
{\bf 364}: 1-17. \DOI{10.1016/j.jmaa.2009.11.014}

\bibitem{WX3}
Q. Wang and C. Xia, Inequalities for eigenvalues of a clamped plate problem. Calc. Var. Partial Differ. Equ., 2011,
{\bf 40}: 273-289. \DOI{10.1007/s00526-010-0340-4}

\bibitem{Xin2}
Y. L. Xin,
Translating soliton of the mean curvature flow. Calc.
Var. Par. Diff. Equs., 2015,  {\bf  54}(2):1995-2016. \DOI{10.1007/s00526-015-0853-y}


\bibitem{Y} H. C. Yang, An estimate of the differance between consecutive eigenvalues.
Preprint IC/91/60 of ICTP, Trieste, 1991.


\bibitem{Z} 
L. Zeng, Eigenvalue inequalities for the clamped plate problem of $\mathfrak{L}^{2}_{\nu}$ operator,  to appear in Science China: Math. https://arxiv.org/abs/2101.07989  
 



\end{thebibliography}
\end{document}